\documentclass[times,10pt]{elsarticle}
\usepackage[margin=1in]{geometry}

\usepackage{amsmath,amssymb,amsthm}
\usepackage{algorithm}
\usepackage{algorithmic}
\usepackage{latexsym}
\usepackage{graphicx}
\usepackage{epstopdf}
\usepackage{color}
\usepackage{hyperref}
\usepackage{multirow}
\usepackage{url}

\theoremstyle{plain}

\newtheorem{theorem}{Theorem}
\newtheorem{lemma}[theorem]{Lemma}
\newtheorem{definition}[theorem]{Definition}
\newtheorem{problem}[theorem]{Problem}
\newtheorem{conjecture}[theorem]{Conjecture}
\newtheorem*{conjecture4M4}{\boldmath The $4M-4$ Conjecture}

\begin{document}

\title{Saving phase: Injectivity and stability for phase retrieval}




\author[princeton]{Afonso S.\ Bandeira}
\author[missouri]{Jameson Cahill}
\author[afit]{Dustin G.\ Mixon}
\author[afit]{Aaron A.\ Nelson}

\address[princeton]{Program in Applied and Computational Mathematics (PACM), Princeton University, Princeton, NJ 08544}
\address[missouri]{Department of Mathematics, University of Missouri, Columbia, MO 65211}
\address[afit]{Department of Mathematics and Statistics, Air Force Institute of Technology, Wright-Patterson Air Force Base, OH 45433}

\begin{abstract}
Recent advances in convex optimization have led to new strides in the phase retrieval problem over finite-dimensional vector spaces.
However, certain fundamental questions remain:
What sorts of measurement vectors uniquely determine every signal up to a global phase factor, and how many are needed to do so?
Furthermore, which measurement ensembles yield stability?
This paper presents several results that address each of these questions.
We begin by characterizing injectivity, and we identify that the complement property is indeed a necessary condition in the complex case.
We then pose a conjecture that $4M-4$ generic measurement vectors are both necessary and sufficient for injectivity in $M$ dimensions, and we prove this conjecture in the special cases where $M=2,3$.
Next, we shift our attention to stability, both in the worst and average cases.
Here, we characterize worst-case stability in the real case by introducing a numerical version of the complement property.
This new property bears some resemblance to the restricted isometry property of compressed sensing and can be used to derive a sharp lower Lipschitz bound on the intensity measurement mapping.
Localized frames are shown to lack this property (suggesting instability), whereas Gaussian random measurements are shown to satisfy this property with high probability.
We conclude by presenting results that use a stochastic noise model in both the real and complex cases, and we leverage Cramer-Rao lower bounds to identify stability with stronger versions of the injectivity characterizations.
\end{abstract}

\begin{keyword}
phase retrieval \sep quantum mechanics \sep bilipschitz function \sep Cramer-Rao lower bound 
\end{keyword}

\maketitle

\section{Introduction}

Signals are often passed through linear systems, and in some applications, only the pointwise absolute value of the output is available for analysis.
For example, in high-power coherent diffractive imaging, this loss of phase information is eminent, as one only has access to the power spectrum of the desired signal~\cite{BunkDPDSSV:07}.
\textit{Phase retrieval} is the problem of recovering a signal from absolute values (squared) of linear measurements, called \textit{intensity measurements}.
Note that phase retrieval is often impossible---intensity measurements with the identity basis effectively discard the phase information of the signal's entries, and so this measurement process is not at all injective; the power spectrum similarly discards the phases of Fourier coefficients.
This fact has led many researchers to invoke a priori knowledge of the desired signal, since intensity measurements might be injective when restricted to a smaller signal class.
Unfortunately, this route has yet to produce practical phase retrieval guarantees, and practitioners currently resort to various ad hoc methods that often fail to work.

Thankfully, there is an alternative approach to phase retrieval, as introduced in 2006 by Balan, Casazza and Edidin~\cite{BalanCE:06}:
Seek injectivity, not by finding a smaller signal class, but rather by designing a larger ensemble of intensity measurements.
In~\cite{BalanCE:06}, Balan et al.\ characterized injectivity in the real case and further leveraged algebraic geometry to show that $4M-2$ intensity measurements suffice for injectivity over $M$-dimensional complex signals.
This realization that so few measurements can yield injectivity has since prompted a flurry of research in search of practical phase retrieval guarantees~\cite{AlexeevBFM:12,Balan:12,BalanBCE:09,CandesESV:11,CandesL:12,CandesSV:11,DemanetH:12,EldarM:12,WaldspurgerAM:12}.
Notably, Cand\`{e}s, Strohmer and Voroninski~\cite{CandesSV:11} viewed intensity measurements as Hilbert-Schmidt inner products between rank-1 operators, and they applied certain intuition from compressed sensing to stably reconstruct the desired $M$-dimensional signal with semidefinite programming using only $\mathcal{O}(M\log M)$ random measurements; similar alternatives and refinements have since been identified~\cite{CandesESV:11,CandesL:12,DemanetH:12,WaldspurgerAM:12}.
Another alternative exploits the polarization identity to discern relative phases between certain intensity measurements; this method uses $\mathcal{O}(M\log M)$ random measurements in concert with an expander graph, and comes with a similar stability guarantee~\cite{AlexeevBFM:12}.

Despite these recent strides in phase retrieval algorithms, there remains a fundamental lack of understanding about what it takes for intensity measurements to be injective, let alone whether measurements yield stability (a more numerical notion of injectivity).
For example, until very recently, it was believed that $3M-2$ intensity measurements sufficed for injectivity (see for example~\cite{CandesESV:11}); this was disproved by Heinosaari, Mazzarella and Wolf~\cite{HeinosaariMW:11}, who used embedding theorems from differential geometry to establish the necessity of $(4+o(1))M$ measurements.
As far as stability is concerned, the most noteworthy achievement to date is due to Eldar and Mendelson~\cite{EldarM:12}, who proved that $\mathcal{O}(M)$ Gaussian random measurements separate distant $M$-dimensional real signals with high probability.
Still, the following problem remains wide open:

\begin{problem}
What are the necessary and sufficient conditions for measurement vectors to yield injective and stable intensity measurements?
\end{problem}

The present paper addresses this problem in a number of ways.
Section~2 focuses on injectivity, and it starts by providing the first known characterization of injectivity in the complex case (Theorem~\ref{thm.complex injective}).
Next, we make a rather surprising identification: that intensity measurements are injective in the complex case precisely when the corresponding phase-only measurements are injective in some sense (Theorem~\ref{thm.complex phase only}).
We then use this identification to prove the necessity of the complement property for injectivity (Theorem~\ref{thm.complex complement property}).
Later, we conjecture that $4M-4$ intensity measurements are necessary and sufficient for injectivity in the complex case, and we prove this conjecture in the cases where $M=2,3$ (Theorems~\ref{thm.4M4 true 2} and~\ref{thm.4M4 true 3}).
Our proof for the $M=3$ case leverages a new test for injectivity, which we then use to verify the injectivity of a certain quantum-mechanics-inspired measurement ensemble, thereby suggesting a new refinement of Wright's conjecture from~\cite{Vogt:78} (see Conjecture~\ref{conjecture.new wrights}). 

We devote Section~3 to stability.
Here, we start by focusing on the real case, for which we give upper and lower Lipschitz bounds of the intensity measurement mapping in terms of singular values of submatrices of the measurement ensemble (Lemma~\ref{lemma.beta expression} and Theorem~\ref{thm.bounding alpha}); this suggests a new matrix condition called the \textit{strong complement property}, which strengthens the complement property of Balan et al.~\cite{BalanCE:06} and bears some resemblance to the restricted isometry property of compressed sensing~\cite{Candes:08}.
As we will discuss, our result corroborates the intuition that localized frames fail to yield stability.
We then show that Gaussian random measurements satisfy the strong complement property with high probability (Theorem~\ref{theorem.random scp bound}), which nicely complements the results of Eldar and Mendelson~\cite{EldarM:12}.
In particular, we find an explicit, intuitive relation between the Lipschitz bounds and the number of intensity measurements per dimension (see Figure~\ref{figure.curves}.
Finally, we present results in both the real and complex cases using a stochastic noise model, much like Balan did for the real case in~\cite{Balan:12}; here, we leverage Cramer-Rao lower bounds to identify stability with stronger versions of the injectivity characterizations (see Theorems~\ref{thm.real random stability} and~\ref{thm.complex random stability}).

\subsection{Notation}

Given a collection of measurement vectors $\Phi=\{\varphi_n\}_{n=1}^N$ in $V=\mathbb{R}^M$ or $\mathbb{C}^M$, consider the intensity measurement process defined by 
\begin{equation*}
(\mathcal{A}(x))(n):=|\langle x,\varphi_n\rangle|^2.
\end{equation*}
Note that $\mathcal{A}(x)=\mathcal{A}(y)$ whenever $y=cx$ for some scalar $c$ of unit modulus.
As such, the mapping $\mathcal{A}\colon V\rightarrow\mathbb{R}^N$ is necessarily not injective.
To resolve this (technical) issue, throughout this paper, we consider sets of the form $V/S$, where $V$ is a vector space and $S$ is a multiplicative subgroup of the field of scalars.
By this notation, we mean to identify vectors $x,y\in V$ for which there exists a scalar $c\in S$ such that $y=cx$; we write $y\equiv x\bmod S$ to convey this identification.
Most (but not all) of the time, $V/S$ is either $\mathbb{R}^M/\{\pm1\}$ or $\mathbb{C}^M/\mathbb{T}$ (here, $\mathbb{T}$ is the complex unit circle), and we view the intensity measurement process as a mapping $\mathcal{A}\colon V/S\rightarrow\mathbb{R}^N$; it is in this way that we will consider the measurement process to be injective or stable.

\section{Injectivity}

\subsection{Injectivity and the complement property}

Phase retrieval is impossible without injective intensity measurements.
In their seminal work on phase retrieval~\cite{BalanCE:06}, Balan, Casazza and Edidin introduce the following property to analyze injectivity:

\begin{definition}
We say $\Phi=\{\varphi_n\}_{n=1}^N$ in $\mathbb{R}^M$ ($\mathbb{C}^M$) satisfies the \textit{complement property (CP)} if for every $S\subseteq\{1,\ldots,N\}$, either $\{\varphi_n\}_{n\in S}$ or $\{\varphi_n\}_{n\in S^\mathrm{c}}$ spans $\mathbb{R}^M$ ($\mathbb{C}^M$).
\end{definition}

In the real case, the complement property is characteristic of injectivity, as demonstrated in~\cite{BalanCE:06}.
We provide the proof of this result below; it contains several key insights which we will apply throughout this paper.

\begin{theorem}
\label{thm.complement property characterization}
Consider $\Phi=\{\varphi_n\}_{n=1}^N\subseteq\mathbb{R}^M$ and the mapping $\mathcal{A}\colon\mathbb{R}^M/\{\pm1\}\rightarrow\mathbb{R}^N$ defined by $(\mathcal{A}(x))(n):=|\langle x,\varphi_n\rangle|^2$. 
Then $\mathcal{A}$ is injective if and only if $\Phi$ satisfies the complement property.
\end{theorem}

\begin{proof}
We will prove both directions by obtaining the contrapositives.

($\Rightarrow$) 
Assume that $\Phi$ is not CP. 
Then there exists $S\subseteq\{1,\ldots,N\}$ such that neither $\{\varphi_n\}_{n\in S}$ nor $\{\varphi_n\}_{n\in S^\mathrm{c}}$ spans $\mathbb{R}^M$. 
This implies that there are nonzero vectors $u,v\in\mathbb{R}^M$ such that $\langle u,\varphi_n\rangle=0$ for all $n\in S$ and $\langle v,\varphi_n\rangle=0$ for all $n\in S^\mathrm{c}$.
For each $n$, we then have
\begin{equation*}
|\langle u\pm v,\varphi_n\rangle|^2
=|\langle u,\varphi_n\rangle|^2\pm2\operatorname{Re}\langle u,\varphi_n\rangle\overline{\langle v,\varphi_n\rangle}+|\langle v,\varphi_n\rangle|^2
=|\langle u,\varphi_n\rangle|^2+|\langle v,\varphi_n\rangle|^2.
\end{equation*}
Since $|\langle u+v,\varphi_n\rangle|^2=|\langle u-v,\varphi_n\rangle|^2$ for every $n$, we have $\mathcal{A}(u+v)=\mathcal{A}(u-v)$. 
Moreover, $u$ and $v$ are nonzero by assumption, and so $u+v\neq\pm(u-v)$.

($\Leftarrow$) 
Assume that $\mathcal{A}$ is not injective. 
Then there exist vectors $x,y\in\mathbb{R}^M$ such that $x\neq\pm y$ and $\mathcal{A}(x)=\mathcal{A}(y)$. 
Taking $S:=\{n:\langle x,\varphi_n\rangle=-\langle y,\varphi_n\rangle\}$, we have $\langle x+y,\varphi_n\rangle=0$ for every $n\in S$.
Otherwise when $n\in S^\mathrm{c}$, we have $\langle x,\varphi_n\rangle=\langle y,\varphi_n\rangle$ and so $\langle x-y,\varphi_n\rangle=0$.
Furthermore, both $x+y$ and $x-y$ are nontrivial since $x\neq\pm y$, and so neither $\{\varphi_n\}_{n\in S}$ nor $\{\varphi_n\}_{n\in S^\mathrm{c}}$ spans $\mathbb{R}^M$.
\qquad
\end{proof}

Note that~\cite{BalanCE:06} erroneously stated that the first part of the above proof also gives that CP is necessary for injectivity in the complex case; the reader is encouraged to spot the logical error.
We wait to identify the error later in this section so as to avoid spoilers.
We will also give a correct proof of the result in question.
In the meantime, let's characterize injectivity in the complex case:

\begin{theorem}
\label{thm.complex injective}
Consider $\Phi=\{\varphi_n\}_{n=1}^N\subseteq\mathbb{C}^M$ and the mapping $\mathcal{A}\colon\mathbb{C}^M/\mathbb{T}\rightarrow\mathbb{R}^N$ defined by $(\mathcal{A}(x))(n):=|\langle x,\varphi_n\rangle|^2$.
Viewing $\{\varphi_n\varphi_n^* u\}_{n=1}^N$ as vectors in $\mathbb{R}^{2M}$, denote $S(u):=\operatorname{span}_\mathbb{R}\{\varphi_n\varphi_n^* u\}_{n=1}^N$.
Then the following are equivalent:
\begin{itemize}
\item[(a)] $\mathcal{A}$ is injective.
\item[(b)] $\operatorname{dim}S(u)\geq 2M-1$ for every $u\in\mathbb{C}^M\setminus\{0\}$.
\item[(c)] $S(u)=\operatorname{span}_\mathbb{R}\{\mathrm{i}u\}^\perp$ for every $u\in\mathbb{C}^M\setminus\{0\}$.
\end{itemize}
\end{theorem}

Before proving this theorem, note that unlike the characterization in the real case, it is not clear whether this characterization can be tested in finite time; instead of being a statement about all (finitely many) partitions of $\{1,\ldots,N\}$, this is a statement about all $u\in\mathbb{C}^M\setminus\{0\}$.
However, we can view this characterization as an analog to the real case in some sense:
In the real case, the complement property is equivalent to having $\operatorname{span}\{\varphi_n\varphi_n^* u\}_{n=1}^N=\mathbb{R}^M$ for all $u\in\mathbb{R}^M\setminus\{0\}$.
As the following proof makes precise, the fact that $\{\varphi_n\varphi_n^* u\}_{n=1}^N$ fails to span all of $\mathbb{R}^{2M}$ is rooted in the fact that more information is lost with phase in the complex case.

\begin{proof}[Proof of Theorem~\ref{thm.complex injective}]
(a) $\Rightarrow$ (c): 
Suppose $\mathcal{A}$ is injective.
We need to show that $\{\varphi_n\varphi_n^* u\}_{n=1}^N$ spans the set of vectors orthogonal to $\mathrm{i}u$.
Here, orthogonality is with respect to the real inner product, which can be expressed as $\langle a,b\rangle_\mathbb{R}=\operatorname{Re}\langle a,b\rangle$.
Note that
\begin{equation*}
|\langle u\pm v,\varphi_n\rangle|^2
=|\langle u,\varphi_n\rangle|^2\pm2\operatorname{Re}\langle u,\varphi_n\rangle\langle \varphi_n,v\rangle+|\langle v,\varphi_n\rangle|^2,
\end{equation*}
and so subtraction gives
\begin{equation}
\label{eq.injectivity to orthogonality}
|\langle u+v,\varphi_n\rangle|^2-|\langle u-v,\varphi_n\rangle|^2
=4\operatorname{Re}\langle u,\varphi_n\rangle\langle \varphi_n,v\rangle
=4\langle \varphi_n\varphi_n^*u,v\rangle_\mathbb{R}.
\end{equation}
In particular, if the right-hand side of \eqref{eq.injectivity to orthogonality} is zero, then injectivity implies that there exists some $\omega$ of unit modulus such that $u+v=\omega(u-v)$.
Since $u\neq0$, we know $\omega\neq-1$, and so rearranging gives
\begin{equation*}
v
=-\frac{1-\omega}{1+\omega}u
=-\frac{(1-\omega)(1+\overline{\omega})}{|1+\omega|^2}u
=-\frac{2\operatorname{Im}\omega}{|1+\omega|^2}~\mathrm{i}u.
\end{equation*}
This means $S(u)^\perp\subseteq\operatorname{span}_\mathbb{R}\{\mathrm{i}u\}$.
To prove $\operatorname{span}_\mathbb{R}\{\mathrm{i}u\}\subseteq S(u)^\perp$, take $v=\alpha \mathrm{i}u$ for some $\alpha\in\mathbb{R}$ and define $\omega:=\frac{1+\alpha \mathrm{i}}{1-\alpha \mathrm{i}}$, which necessarily has unit modulus.
Then
\begin{equation*}
u+v
=u+\alpha \mathrm{i} u
=(1+\alpha \mathrm{i})u
=\frac{1+\alpha \mathrm{i}}{1-\alpha \mathrm{i}}(u-\alpha \mathrm{i}u)
=\omega(u-v).
\end{equation*}
Thus, the left-hand side of \eqref{eq.injectivity to orthogonality} is zero, meaning $v\in S(u)^\perp$.

(b) $\Leftrightarrow$ (c):
First, (b) immediately follows from (c).
For the other direction, note that $\mathrm{i}u$ is necessarily orthogonal to every $\varphi_n\varphi_n^* u$:
\begin{equation*}
\langle \varphi_n\varphi_n^* u,\mathrm{i}u\rangle_\mathbb{R}
=\operatorname{Re}\langle \varphi_n\varphi_n^* u,\mathrm{i}u\rangle
=\operatorname{Re}\langle u,\varphi_n\rangle\langle \varphi_n,\mathrm{i}u\rangle
=-\operatorname{Re}\mathrm{i}|\langle u,\varphi_n\rangle|^2
=0.
\end{equation*}
Thus, $\operatorname{span}_\mathbb{R}\{\mathrm{i}u\}\subseteq S(u)^\perp$, and by (b), $\operatorname{dim}S(u)^\perp\leq1$, both of which gives (c).

(c) $\Rightarrow$ (a):
This portion of the proof is inspired by Mukherjee's analysis in~\cite{Mukherjee:81}.
Suppose $\mathcal{A}(x)=\mathcal{A}(y)$.
If $x=y$, we are done.
Otherwise, $x-y\neq0$, and so we may apply (c) to $u=x-y$.
First, note that 
\begin{equation*}
\langle \varphi_n\varphi_n^*(x-y),x+y\rangle_\mathbb{R}
=\operatorname{Re}\langle \varphi_n\varphi_n^*(x-y),x+y\rangle
=\operatorname{Re}(x+y)^*\varphi_n\varphi_n^*(x-y),
\end{equation*}
and so expanding gives
\begin{equation*}
\langle \varphi_n\varphi_n^*(x-y),x+y\rangle_\mathbb{R}
=\operatorname{Re}\Big(|\varphi_n^*x|^2-x^*\varphi_n\varphi_n^*y+y^*\varphi_n\varphi_n^*x-|\varphi_n^*y|^2\Big)
=\operatorname{Re}\Big(-x^*\varphi_n\varphi_n^*y+\overline{x^*\varphi_n\varphi_n^*y}\Big)
=0.
\end{equation*}
Since $x+y\in S(x-y)^\perp=\operatorname{span}_\mathbb{R}\{\mathrm{i}(x-y)\}$, there exists $\alpha\in\mathbb{R}$ such that $x+y=\alpha \mathrm{i}(x-y)$, and so rearranging gives $y=\frac{1-\alpha \mathrm{i}}{1+\alpha \mathrm{i}}x$, meaning $y\equiv x\bmod\mathbb{T}$.
\qquad
\end{proof}

The above theorem leaves a lot to be desired; it is still unclear what it takes for a complex ensemble to yield injective intensity measurements.
While in pursuit of a more clear understanding, we established the following bizarre characterization:
A complex ensemble yields injective intensity measurements precisely when it yields injective \textit{phase-only} measurements (in some sense).
This is made more precise in the following theorem statement:

\begin{theorem}
\label{thm.complex phase only}
Consider $\Phi=\{\varphi_n\}_{n=1}^N\subseteq\mathbb{C}^M$ and the mapping $\mathcal{A}\colon\mathbb{C}^M/\mathbb{T}\rightarrow\mathbb{R}^N$ defined by $(\mathcal{A}(x))(n):=|\langle x,\varphi_n\rangle|^2$.
Then $\mathcal{A}$ is injective if and only if the following statement holds: 
If for every $n=1,\ldots,N$, either $\operatorname{arg}(\langle x,\varphi_n\rangle^2)=\operatorname{arg}(\langle y,\varphi_n\rangle^2)$ or one of the sides is not well-defined, then $x=0$, $y=0$, or $y\equiv x\bmod\mathbb{R}\setminus\{0\}$. 
\end{theorem}

\begin{proof}
By Theorem~\ref{thm.complex injective}, $\mathcal{A}$ is injective if and only if
\begin{equation}
\label{eq.injective equivalence 1}
\forall x\in\mathbb{C}^M\setminus\{0\}, \qquad\operatorname{span}_\mathbb{R}\{\varphi_n\varphi_n^*x\}_{n=1}^N=\operatorname{span}_\mathbb{R}\{\mathrm{i}x\}^{\perp}.
\end{equation}
Taking orthogonal complements of both sides, note that regardless of $x\in\mathbb{C}^M\setminus\{0\}$, we know $\operatorname{span}_\mathbb{R}\{\mathrm{i}x\}$ is necessarily a subset of $(\operatorname{span}_\mathbb{R}\{\varphi_n\varphi_n^*x\}_{n=1}^N)^\perp$, and so \eqref{eq.injective equivalence 1} is equivalent to
\begin{equation*}
\forall x\in\mathbb{C}^M\setminus\{0\},\qquad\operatorname{Re}\langle\varphi_n\varphi_n^*x,\mathrm{i}y\rangle=0\quad\forall n=1,\ldots,N\quad\Longrightarrow\quad y=0 \text{ or } y\equiv x\bmod\mathbb{R}\setminus\{0\}.
\end{equation*}
Thus, we need to determine when $\operatorname{Im}\langle x,\varphi_n\rangle\overline{\langle y,\varphi_n\rangle}=\operatorname{Re}\langle\varphi_n\varphi_n^*x,\mathrm{i}y\rangle=0$.
We claim that this is true if and only if $\operatorname{arg}(\langle x,\varphi_n\rangle^2)=\operatorname{arg}(\langle y,\varphi_n\rangle^2)$ or one of the sides is not well-defined.
To see this, we substitute $a:=\langle x,\varphi_n\rangle$ and $b:=\langle y,\varphi_n\rangle$.
Then to complete the proof, it suffices to show that $\operatorname{Im}a\overline{b}=0$ if and only if $\operatorname{arg}(a^2)=\operatorname{arg}(b^2)$, $a=0$, or $b=0$.

($\Leftarrow$)
If either $a$ or $b$ is zero, the result is immediate.
Otherwise, if $2\operatorname{arg}(a)=\operatorname{arg}(a^2)=\operatorname{arg}(b^2)=2\operatorname{arg}(b)$, then $2\pi$ divides $2(\operatorname{arg}(a)-\operatorname{arg}(b))$, and so $\operatorname{arg}(a\overline{b})=\operatorname{arg}(a)-\operatorname{arg}(b)$ is a multiple of $\pi$.
This implies that $a\overline{b}\in\mathbb{R}$, and so $\operatorname{Im}a\overline{b}=0$.

($\Rightarrow$)
Suppose $\operatorname{Im}a\overline{b}=0$.
Taking the polar decompositions $a=re^{\mathrm{i}\theta}$ and $b=se^{\mathrm{i}\phi}$, we equivalently have that $rs\sin{(\theta-\phi)}=0$.
Certainly, this can occur whenever $r$ or $s$ is zero, i.e., $a=0$ or $b=0$.
Otherwise, a difference formula then gives $\sin{\theta}\cos{\phi}=\cos{\theta}\sin{\phi}$.
From this, we know that if $\theta$ is an integer multiple of $\pi/2$, then $\phi$ is as well, and vice versa, in which case $\operatorname{arg}(a^2)=2\operatorname{arg}(a)=\pi=2\operatorname{arg}(b)=\operatorname{arg}(b^2)$.
Else, we can divide both sides by $\cos{\theta}\cos{\phi}$ to obtain $\tan{\theta}=\tan{\phi}$, from which it is evident that $\theta\equiv\phi\bmod\pi$, and so $\operatorname{arg}(a^2)=2\operatorname{arg}(a)=2\operatorname{arg}(b)=\operatorname{arg}(b^2)$.
\qquad
\end{proof}

To be clear, it is unknown to the authors whether such phase-only measurements arrive in any application (nor whether a corresponding reconstruction algorithm is feasible), but we find it rather striking that injectivity in this setting is equivalent to injectivity in ours.
We will actually use this result to (correctly) prove the necessity of CP for injectivity.
First, we need the following lemma, which is interesting in its own right:

\begin{lemma}
\label{lem.injectivity of B}
Consider $\Phi=\{\varphi_n\}_{n=1}^N\subseteq\mathbb{C}^M$ and the mapping $\mathcal{A}\colon\mathbb{C}^M/\mathbb{T}\rightarrow\mathbb{R}^N$ defined by $(\mathcal{A}(x))(n):=|\langle x,\varphi_n\rangle|^2$.
If $\mathcal{A}$ is injective, then the mapping $\mathcal{B}\colon\mathbb{C}^M/\{\pm1\}\rightarrow\mathbb{C}^N$ defined by $(\mathcal{B}(x))(n):=\langle x,\varphi_n\rangle^2$ is also injective.
\end{lemma}

\begin{proof}
Suppose $\mathcal{A}$ is injective.
Then we have the following facts (one by definition, and the other by Theorem~\ref{thm.complex phase only}):
\begin{itemize}
\item[(i)] If $\forall n=1,\ldots,N$, $|\langle x,\varphi_n\rangle|^2=|\langle y,\varphi_n\rangle|^2$, then $y\equiv x\bmod\mathbb{T}$.
\item[(ii)] If $\forall n=1,\ldots,N$, either $\operatorname{arg}(\langle x,\varphi_n\rangle^2)=\operatorname{arg}(\langle y,\varphi_n\rangle^2)$ or one of the sides is not well-defined, then $x=0$, $y=0$, or $y\equiv x\bmod\mathbb{R}\setminus\{0\}$.
\end{itemize}
Now suppose we have $\langle x,\varphi_n\rangle^2=\langle y,\varphi_n\rangle^2$ for all $n=1,\ldots,N$.
Then their moduli and arguments are also equal, and so (i) and (ii) both apply.
Of course, $y\equiv x\bmod\mathbb{T}$ implies $x=0$ if and only if $y=0$.
Otherwise both are nonzero, in which case there exists $\omega\in\mathbb{T}\cap\mathbb{R}\setminus\{0\}=\{\pm1\}$ such that $y=\omega x$.
In either case, $y\equiv x\bmod\{\pm1\}$, so $\mathcal{B}$ is injective.
\qquad
\end{proof}

\begin{theorem}
\label{thm.complex complement property}
Consider $\Phi=\{\varphi_n\}_{n=1}^N\subseteq\mathbb{C}^M$ and the mapping $\mathcal{A}\colon\mathbb{C}^M/\mathbb{T}\rightarrow\mathbb{R}^N$ defined by $(\mathcal{A}(x))(n):=|\langle x,\varphi_n\rangle|^2$.
If $\mathcal{A}$ is injective, then $\Phi$ satisfies the complement property.
\end{theorem}

Before giving the proof, let's first divulge why the first part of the proof of Theorem~\ref{thm.complement property characterization} does not suffice: 
It demonstrates that $u+v\neq\pm(u-v)$, but fails to establish that $u+v\not\equiv u-v\bmod\mathbb{T}$; for instance, it could very well be the case that $u+v=\mathrm{i}(u-v)$, and so injectivity would not be violated \textit{in the complex case}.
Regardless, the following proof, which leverages the injectivity of $\mathcal{B}$ modulo $\{\pm1\}$, resolves this issue.

\begin{proof}[Proof of Theorem~\ref{thm.complex complement property}]
Recall that if $\mathcal{A}$ is injective, then so is the mapping $\mathcal{B}$ of Lemma~\ref{lem.injectivity of B}.
Therefore, it suffices to show that $\Phi$ is CP if $\mathcal{B}$ is injective.
To complete the proof, we will obtain the contrapositive (note the similarity to the proof of Theorem~\ref{thm.complement property characterization}).
Suppose $\Phi$ is not CP. 
Then there exists $S\subseteq\{1,\ldots,N\}$ such that neither $\{\varphi_n\}_{n\in S}$ nor $\{\varphi_n\}_{n\in S^\mathrm{c}}$ spans $\mathbb{C}^M$. 
This implies that there are nonzero vectors $u,v\in\mathbb{C}^M$ such that $\langle u,\varphi_n\rangle=0$ for all $n\in S$ and $\langle v,\varphi_n\rangle=0$ for all $n\in S^\mathrm{c}$.
For each $n$, we then have
\begin{equation*}
\langle u\pm v,\varphi_n\rangle^2
=\langle u,\varphi_n\rangle^2\pm2\langle u,\varphi_n\rangle\langle v,\varphi_n\rangle+\langle v,\varphi_n\rangle^2
=\langle u,\varphi_n\rangle^2+\langle v,\varphi_n\rangle^2.
\end{equation*}
Since $\langle u+v,\varphi_n\rangle^2=\langle u-v,\varphi_n\rangle^2$ for every $n$, we have $\mathcal{B}(u+v)=\mathcal{B}(u-v)$. 
Moreover, $u$ and $v$ are nonzero by assumption, and so $u+v\neq\pm(u-v)$.
\qquad
\end{proof}

Note that the complement property is necessary but not sufficient for injectivity.  
To see this, consider measurement vectors $(1,0)$, $(0,1)$ and $(1,1)$.  These certainly satisfy the complement property, but $\mathcal{A}((1,\mathrm{i}))=(1,1,2)=\mathcal{A}((1,-\mathrm{i}))$, despite the fact that $(1,\mathrm{i})\not\equiv(1,-\mathrm{i})\bmod\mathbb{T}$; in general, real measurement vectors fail to yield injective intensity measurements in the complex setting since they do not distinguish complex conjugates.
Indeed, we have yet to find a ``good'' sufficient condition for injectivity in the complex case.
As an analogy for what we really want, consider the notion of \textit{full spark}:
An ensemble $\{\varphi_n\}_{n=1}^N\subseteq\mathbb{R}^M$ is said to be full spark if every subcollection of $M$ vectors spans $\mathbb{R}^M$.
It is easy to see that full spark ensembles with $N\geq 2M-1$ necessarily satisfy the complement property (thereby implying injectivity in the real case), and furthermore, the notion of full spark is simple enough to admit deterministic constructions~\cite{AlexeevCM:12,PuschelK:05}.
Deterministic measurement ensembles are particularly desirable for the complex case, and so finding a good sufficient condition for injectivity is an important problem that remains open.

\subsection{Towards a rank-nullity theorem for phase retrieval}

If you think of a matrix $\Phi$ as being built one column at a time, then the rank-nullity theorem states that each column contributes to either the column space or the null space.
If the columns are then used as linear measurement vectors (say we take measurements $y=\Phi^*x$ of a vector $x$), then the column space of $\Phi$ gives the subspace that is actually sampled, and the null space captures the algebraic nature of the measurements' redundancy.
Therefore, an efficient sampling of an entire vector space would apply a matrix $\Phi$ with a small null space and large column space (e.g., an invertible square matrix).
How do we find such a sampling with intensity measurements?
The following makes this question more precise:
\
\begin{problem}
\label{prob.n star}
For any dimension $M$, what is the smallest number $N^*(M)$ of injective intensity measurements, and how do we design such measurement vectors?
\end{problem}

To be clear, this problem was completely solved in the real case by Balan, Casazza and Edidin~\cite{BalanCE:06}.
Indeed, Theorem~\ref{thm.complement property characterization} immediately implies that $2M-2$ intensity measurements are necessarily not injective, and furthermore that $2M-1$ measurements are injective if and only if the measurement vectors are full spark. 
As such, we will focus our attention to the complex case.

In the complex case, Problem~\ref{prob.n star} has some history in the quantum mechanics literature.
For example, \cite{Vogt:78} presents \textit{Wright's conjecture} that three observables suffice to uniquely determine any pure state.
In phase retrieval parlance, the conjecture states that there exist unitary matrices $U_1$, $U_2$ and $U_3$ such that $\Phi=[U_1~U_2~U_3]$ yields injective intensity measurements (here, the measurement vectors are the columns of $\Phi$).
Note that Wright's conjecture actually implies that $N^*(M)\leq 3M-2$; indeed, $U_1$ determines the norm (squared) of the signal, rendering the last column of both $U_2$ and $U_3$ unnecessary.
Finkelstein~\cite{Finkelstein:04} later proved that $N^*(M)\geq 3M-2$; combined with Wright's conjecture, this led many to believe that $N^*(M)=3M-2$ (for example, see~\cite{CandesESV:11}).
However, both this and Wright's conjecture were recently disproved in~\cite{HeinosaariMW:11}, in which Heinosaari, Mazzarella and Wolf invoked embedding theorems from differential geometry to prove that 
\begin{equation}
\label{eq.HMW lower bound}
N^*(M)\geq
\left\{
\begin{array}{ll}
4M-2\alpha(M-1)-3 &\mbox{for all } M\\
4M-2\alpha(M-1)-2 &\mbox{if } M\mbox{ is odd and }\alpha(M-1)=2\bmod 4\\
4M-2\alpha(M-1)-1 &\mbox{if } M\mbox{ is odd and }\alpha(M-1)=3\bmod 4,
\end{array}
\right.
\end{equation}
where $\alpha(M-1)\leq\log_2(M)$ is the number of $1$'s in the binary representation of $M-1$.
By comparison, Balan, Casazza and Edidin~\cite{BalanCE:06} proved that $N^*(M)\leq4M-2$, and so we at least have the asymptotic expression $N^*(M)=(4+o(1))M$.

At this point, we should clarify some intuition for $N^*(M)$ by explaining the nature of these best known lower and upper bounds.
First, the lower bound~\eqref{eq.HMW lower bound} follows from an older result that complex projective space $\mathbb{C}\mathbf{P}^n$ does not smoothly embed into $\mathbb{R}^{4n-2\alpha(n)}$ (and other slight refinements which depend on $n$); this is due to Mayer~\cite{Mayer:65}, but we highly recommend James's survey on the topic~\cite{James:71}.
To prove~\eqref{eq.HMW lower bound} from this, suppose $\mathcal{A}\colon\mathbb{C}^M/\mathbb{T}\rightarrow\mathbb{R}^N$ were injective.
Then $\mathcal{E}$ defined by $\mathcal{E}(x):=\mathcal{A}(x)/\|x\|^2$ embeds $\mathbb{C}\mathbf{P}^{M-1}$ into $\mathbb{R}^N$, and as Heinosaari et al.\ show, the embedding is necessarily smooth; considering $\mathcal{A}(x)$ is made up of rather simple polynomials, the fact that $\mathcal{E}$ is smooth should not come as a surprise.
As such, the nonembedding result produces the best known lower bound.
To evaluate this bound, first note that Milgram~\cite{Milgram:67} constructs an embedding of $\mathbb{C}\mathbf{P}^n$ into $\mathbb{R}^{4n-\alpha(n)+1}$, establishing the importance of the $\alpha(n)$ term, but the constructed embedding does not correspond to an intensity measurement process.
In order to relate these embedding results to our problem, consider the real case: 
It is known that for odd $n\geq7$, real projective space $\mathbb{R}\mathbf{P}^n$ smoothly embeds into $\mathbb{R}^{2n-\alpha(n)+1}$~\cite{Steer:70}, which means the analogous lower bound for the real case would necessarily be smaller than $2(M-1)-\alpha(M-1)+1=2M-\alpha(M-1)-1<2M-1$.
This indicates that the $\alpha(M-1)$ term in~\eqref{eq.HMW lower bound} might be an artifact of the proof technique, rather than of $N^*(M)$.

There is also some intuition to be gained from the upper bound $N^*(M)\leq4M-2$, which Balan et al.\ proved by applying certain techniques from algebraic geometry (some of which we will apply later in this section).
In fact, their result actually gives that $4M-2$ or more measurement vectors, if chosen \textit{generically}, will yield injective intensity measurements; here, generic is a technical term involving the Zariski topology, but it can be thought of as some undisclosed property which is satisfied with probability 1 by measurement vectors drawn from continuous distributions.
This leads us to think that $N^*(M)$ generic measurement vectors might also yield injectivity.

The lemma that follows will help to refine our intuition for $N^*(M)$, and it will also play a key role in the main theorems of this section (a similar result appears in~\cite{HeinosaariMW:11}).
Before stating the result, define the real $M^2$-dimensional space $\mathbb{H}^{M\times M}$ of self-adjoint $M\times M$ matrices; note that this is not a vector space over the complex numbers since the diagonal of a self-adjoint matrix must be real.
Given an ensemble of measurement vectors $\{\varphi_n\}_{n=1}^N\subseteq\mathbb{C}^M$, 
define the \textit{super analysis operator} $\mathbf{A}\colon\mathbb{H}^{M\times M}\rightarrow\mathbb{R}^N$ by $(\mathbf{A}H)(n)=\langle H,\varphi_n\varphi_n^*\rangle_\mathrm{HS}$; here, $\langle\cdot,\cdot\rangle_\mathrm{HS}$ denotes the Hilbert-Schmidt inner product, which induces the Frobenius matrix norm.
Note that $\mathbf{A}$ is a linear operator, and yet 
\begin{equation*}
(\mathbf{A}xx^*)(n)
=\langle xx^*,\varphi_n\varphi_n^*\rangle_\mathrm{HS}
=\operatorname{Tr}[\varphi_n\varphi_n^*xx^*]
=\operatorname{Tr}[\varphi_n^*xx^*\varphi_n]
=\varphi_n^*xx^*\varphi_n
=|\langle x,\varphi_n\rangle|^2
=(\mathcal{A}(x))(n).
\end{equation*}
In words, the class of vectors identified with $x$ modulo $\mathbb{T}$ can be ``lifted'' to $xx^*$, thereby linearizing the intensity measurement process at the price of squaring the dimension of the vector space of interest; this identification has been exploited by some of the most noteworthy strides in modern phase retrieval~\cite{BalanBCE:09,CandesSV:11}.
As the following lemma shows, this identification can also be used to characterize injectivity:

\begin{lemma}
\label{lem.not injective rank 2}
$\mathcal{A}$ is not injective if and only if there exists a matrix of rank $1$ or $2$ in the null space of $\mathbf{A}$.
\end{lemma}

\begin{proof}
($\Rightarrow$)
If $\mathcal{A}$ is not injective, then there exist $x,y\in\mathbb{C}^M/\mathbb{T}$ with $x\nequiv y\bmod\mathbb{T}$ such that $\mathcal{A}(x)=\mathcal{A}(y)$.
That is, $\mathbf{A}xx^*=\mathbf{A}yy^*$, and so $xx^*-yy^*$ is in the null space of $\mathbf{A}$.

($\Leftarrow$)
First, suppose there is a rank-$1$ matrix $H$ in the null space of $\mathbf{A}$.
Then there exists $x\in\mathbb{C}^M$ such that $H=xx^*$ and $(\mathcal{A}(x))(n)=(\mathbf{A}xx^*)(n)=0=(\mathcal{A}(0))(n)$.
But $x\not\equiv0\bmod\mathbb{T}$, and so $\mathcal{A}$ is not injective.
Now suppose there is a rank-$2$ matrix $H$ in the null space of $\mathbf{A}$.
Then by the spectral theorem, there are orthonormal $u_1,u_2\in\mathbb{C}^M$ and nonzero $\lambda_1\geq\lambda_2$ such that $H=\lambda_1u_1u_1^*+\lambda_2u_2u_2^*$.
Since $H$ is in the null space of $\mathbf{A}$, the following holds for every $n$:
\begin{equation}
\label{eq.rank 2 injective}
0
=\langle H,\varphi_n\varphi_n^*\rangle_\mathrm{HS}
=\langle \lambda_1u_1u_1^*+\lambda_2u_2u_2^*,\varphi_n\varphi_n^*\rangle_\mathrm{HS}
=\lambda_1|\langle u_1,\varphi_n\rangle|^2+\lambda_2|\langle u_2,\varphi_n\rangle|^2.
\end{equation}
Taking $x:=|\lambda_1|^{1/2}u_1$ and $y:=|\lambda_2|^{1/2}u_2$, note that $y\not\equiv x\bmod\mathbb{T}$ since they are nonzero and orthogonal.
We claim that $\mathcal{A}(x)=\mathcal{A}(y)$, which would complete the proof.
If $\lambda_1$ and $\lambda_2$ have the same sign, then by \eqref{eq.rank 2 injective}, $|\langle x,\varphi_n\rangle|^2+|\langle y,\varphi_n\rangle|^2=0$ for every $n$, meaning $|\langle x,\varphi_n\rangle|^2=0=|\langle y,\varphi_n\rangle|^2$.
Otherwise, $\lambda_1>0>\lambda_2$, and so $xx^*-yy^*=\lambda_1u_1u_1^*+\lambda_2u_2u_2^*=A$ is in the null space of $\mathbf{A}$, meaning $\mathcal{A}(x)=\mathbf{A}xx^*=\mathbf{A}yy^*=\mathcal{A}(y)$.
\qquad
\end{proof}

Lemma~\ref{lem.not injective rank 2} indicates that we want the null space of $\mathbf{A}$ to avoid nonzero matrices of rank $\leq 2$.
Intuitively, this is easier when the ``dimension'' of this set of matrices is small.
To get some idea of this dimension, let's count real degrees of freedom.
By the spectral theorem, almost every matrix in $\mathbb{H}^{M\times M}$ of rank $\leq 2$ can be uniquely expressed as $\lambda_1u_1u_1^*+\lambda_2u_2u_2^*$ with $\lambda_1\leq\lambda_2$.
Here, $(\lambda_1,\lambda_2)$ has two degrees of freedom.
Next, $u_1$ can be any vector in $\mathbb{C}^M$, except its norm must be $1$.
Also, since $u_1$ is only unique up to global phase, we take its first entry to be nonnegative without loss of generality.
Given the norm and phase constraints, $u_1$ has a total of $2M-2$ real degrees of freedom.
Finally, $u_2$ has the same norm and phase constraints, but it must also be orthogonal to $u_1$, that is, $\operatorname{Re}\langle u_2,u_1\rangle=\operatorname{Im}\langle u_2,u_1\rangle=0$.
As such, $u_2$ has $2M-4$ real degrees of freedom.
All together, we can expect the set of matrices in question to have $2+(2M-2)+(2M-4)=4M-4$ real dimensions.

If the set $S$ of matrices of rank $\leq 2$ formed a subspace of $\mathbb{H}^{M\times M}$ (it doesn't), then we could expect the null space of $\mathbf{A}$ to intersect that subspace nontrivially whenever $\dim\operatorname{null}(\mathbf{A})+(4M-4)>\dim(\mathbb{H}^{M\times M})=M^2$.
By the rank-nullity theorem, this would indicate that injectivity requires 
\begin{equation}
\label{eq.4M-4 reasoning}
N
\geq\operatorname{rank}(\mathbf{A})
=M^2-\dim\operatorname{null}(\mathbf{A})
\geq 4M-4.
\end{equation}
Of course, this logic is not technically valid since $S$ is not a subspace.
It is, however, a special kind of set: a real projective variety.
To see this, let's first show that it is a real algebraic variety, specifically, the set of members of $\mathbb{H}^{M\times M}$ for which all $3\times 3$ minors are zero.
Of course, every member of $S$ has this minor property.
Next, we show that members of $S$ are the only matrices with this property: 
If the rank of a given matrix is $\geq3$, then it has an $M\times 3$ submatrix of linearly independent columns, and since the rank of its transpose is also $\geq3$, this $M\times 3$ submatrix must have $3$ linearly independent rows, thereby implicating a full-rank $3\times 3$ submatrix.
This variety is said to be projective because it is closed under scalar multiplication. 
If $S$ were a projective variety over an algebraically closed field (it's not), then the projective dimension theorem (Theorem 7.2 of~\cite{Hartshorne:77}) says that $S$ intersects $\operatorname{null}(\mathbf{A})$ nontrivially whenever the dimensions are large enough: $\dim\operatorname{null}(\mathbf{A})+\dim S>\dim\mathbb{H}^{M\times M}$, thereby implying that injectivity requires \eqref{eq.4M-4 reasoning}.
Unfortunately, this theorem is not valid when the field is $\mathbb{R}$;
for example, the cone defined by $x^2+y^2-z^2=0$ in $\mathbb{R}^3$ is a projective variety of dimension $2$, but its intersection with the $2$-dimensional $xy$-plane is trivial, despite the fact that $2+2>3$.

In the absence of a proof, we pose the natural conjecture:

\begin{conjecture4M4}
Consider $\Phi=\{\varphi_n\}_{n=1}^N\subseteq\mathbb{C}^M$ and the mapping $\mathcal{A}\colon\mathbb{C}^M/\mathbb{T}\rightarrow\mathbb{R}^N$ defined by $(\mathcal{A}(x))(n):=|\langle x,\varphi_n\rangle|^2$.
If $M\geq2$, then the following statements hold:
\begin{itemize}
\item[(a)] If $N<4M-4$, then $\mathcal{A}$ is not injective.
\item[(b)] If $N\geq4M-4$, then $\mathcal{A}$ is injective for generic $\Phi$.
\end{itemize}
\end{conjecture4M4}

For the sake of clarity, we now explicitly state what is meant by the word ``generic.''
As indicated above, a real algebraic variety is the set of common zeros of a finite set of polynomials with real coefficients.
Taking all such varieties in $\mathbb{R}^n$ to be closed sets defines the \textit{Zariski topology} on $\mathbb{R}^n$.
Viewing $\Phi$ as a member of $\mathbb{R}^{2MN}$, then we say a \textit{generic} $\Phi$ is any member of some undisclosed nonempty Zariski-open subset of $\mathbb{R}^{2MN}$.
Considering Zariski-open set are either empty or dense with full measure, genericity is a particularly strong property.
As such, another way to state part (b) of the $4M-4$ conjecture is ``If $N\geq4M-4$, then there exists a real algebraic variety $V\subseteq\mathbb{R}^{2MN}$ such that $\mathcal{A}$ is injective for every $\Phi\not\in V$.''
Note that the work of Balan, Casazza and Edidin~\cite{BalanCE:06} already proves this for $N\geq 4M-2$.
Also note that the analogous statement of (b) holds in the real case:
Full spark measurement vectors are generic, and they satisfy the complement property whenever $N\geq 2M-1$.

At this point, it is fitting to mention that after we initially formulated this conjecture, Bodmann presented a Vandermonde construction of $4M-4$ injective intensity measurements at a phase retrieval workshop at the Erwin Schr\"{o}dinger International Institute for Mathematical Physics.
The result has since been documented in~\cite{BodmannH:12}, and it establishes one consequence of the $4M-4$ conjecture: $N^*(M)\leq 4M-4$.

As incremental progress toward solving the $4M-4$ conjecture, we offer the following result:

\begin{theorem}
\label{thm.4M4 true 2}
The $4M-4$ Conjecture is true when $M=2$.
\end{theorem}

\begin{proof}
(a) 
Since $\mathbf{A}$ is a linear map from $4$-dimensional real space to $N$-dimensional real space, the null space of $\mathbf{A}$ is necessarily nontrivial by the rank-nullity theorem.
Furthermore, every nonzero member of this null space has rank $1$ or $2$, and so Lemma~\ref{lem.not injective rank 2} gives that $\mathcal{A}$ is not injective.

(b)
Consider the following matrix formed by 16 real variables:
\begin{equation}
\label{eq.variable matrix}
\Phi(x)=
\left[\begin{array}{cccc}
x_1+\mathrm{i}x_2 & x_5+\mathrm{i}x_6 & x_9+\mathrm{i}x_{10} & x_{13}+\mathrm{i}x_{14} \\
x_3+\mathrm{i}x_4 & x_7+\mathrm{i}x_8 & x_{11}+\mathrm{i}x_{12} & x_{15}+\mathrm{i}x_{16} 
\end{array}\right].
\end{equation}
If we denote the $n$th column of $\Phi(x)$ by $\varphi_n(x)$, then we have that $\mathcal{A}$ is injective precisely when $x\in\mathbb{R}^{16}$ produces a basis $\{\varphi_n(x)\varphi_n(x)^*\}_{n=1}^{4}$ for the space of $2\times 2$ self-adjoint operators.
Indeed, in this case $zz^*$ is uniquely determined by $\mathbf{A}zz^*=\{\langle zz^*,\varphi_n(x)\varphi_n(x)^*\rangle_\mathrm{HS}\}_{n=1}^{4}=\mathcal{A}(z)$, which in turn determines $z$ up to a global phase factor.
Let $\mathbf{A}(x)$ be the $4\times 4$ matrix representation of the super analysis operator, whose $n$th row gives the coordinates of $\varphi_n(x)\varphi_n(x)^*$ in terms of some basis for $\mathbb{H}^{2\times 2}$, say
\begin{equation}
\label{eq.onb for h}
\left\{
\left[\begin{array}{rr}
1 & 0 \\ 0 & 1
\end{array}\right],
\left[\begin{array}{rr}
0 & 0 \\ 0 & 1
\end{array}\right],
\frac{1}{\sqrt{2}}\left[\begin{array}{rr}
0 & 1 \\ 1 & 0
\end{array}\right],
\frac{1}{\sqrt{2}}\left[\begin{array}{rr}
0 & \mathrm{i} \\ -\mathrm{i} & 0
\end{array}\right]
\right\}.
\end{equation}
Then $V=\{x:\operatorname{Re}\det\mathbf{A}(x)=\operatorname{Im}\det\mathbf{A}(x)=0\}$ is a real algebraic variety in $\mathbb{R}^{16}$, and we see that $\mathcal{A}$ is injective whenever $x\in V^\mathrm{c}$. 
Since $V^\mathrm{c}$ is Zariski-open, it is either empty or dense with full measure. 
In fact, $V^\mathrm{c}$ is not empty, since we may take $x$ such that 
\begin{equation*}
\Phi(x)
=\left[
\begin{array}{cccc}
1&0&1&1\\
0&1&1&\mathrm{i}
\end{array}
\right],
\end{equation*}
as indicated in Theorem~4.1 of~\cite{BalanBCE:07}.
Therefore, $V^\mathrm{c}$ is dense with full measure.
\qquad
\end{proof}

We also have a proof for the $M=3$ case, but we first introduce Algorithm~\ref{algorithm}, namely the \textit{HMW test} for injectivity; we name it after Heinosaari, Mazarella and Wolf, who implicitly introduce this algorithm in their paper~\cite{HeinosaariMW:11}.

\begin{algorithm}[h]
\label{algorithm}
\caption{The HMW test for injectivity when $M=3$}
\textbf{Input:} Measurement vectors $\{\varphi_n\}_{n=1}^N\subseteq\mathbb{C}^3$\\
\textbf{Output:} Whether $\mathcal{A}$ is injective
\begin{algorithmic}
\STATE Define $\mathbf{A}\colon\mathbb{H}^{3\times 3}\rightarrow\mathbb{R}^N$ such that $\mathbf{A}H=\{\langle H,\varphi_n\varphi_n^*\rangle_\textrm{HS}\}_{n=1}^N$ \hfill \COMMENT{assemble the super analysis operator}
\IF{$\operatorname{dim}\operatorname{null}(\mathbf{A})=0$}
\STATE ``INJECTIVE'' \hfill \COMMENT{if $\mathbf{A}$ is injective, then $\mathcal{A}$ is injective}
\ELSE
\STATE Pick $H\in\operatorname{null}(\mathbf{A})$, $H\neq0$
\IF{$\operatorname{dim}\operatorname{null}(\mathbf{A})=1$ and $\det(H)\neq0$}
\STATE ``INJECTIVE'' \hfill \COMMENT{if $\mathbf{A}$ only maps nonsingular matrices to zero, then $\mathcal{A}$ is injective}
\ELSE
\STATE ``NOT INJECTIVE'' \hfill \COMMENT{in the remaining case, $\mathbf{A}$ maps differences of rank-$1$ matrices to zero}
\ENDIF
\ENDIF
\end{algorithmic}
\end{algorithm}

\begin{theorem}[cf.\ Proposition 6 in~\cite{HeinosaariMW:11}]
When $M=3$, the HMW test correctly determines whether $\mathcal{A}$ is injective.
\end{theorem}

\begin{proof}
First, if $\mathbf{A}$ is injective, then $\mathcal{A}(x)=\mathbf{A}xx^*=\mathbf{A}yy^*=\mathcal{A}(y)$ if and only if $xx^*=yy^*$, i.e., $y\equiv x\bmod\mathbb{T}$.
Next, suppose $\mathbf{A}$ has a $1$-dimensional null space.
Then Lemma~\ref{lem.not injective rank 2} gives that $\mathcal{A}$ is injective if and only if the null space of $\mathbf{A}$ is spanned by a matrix of full rank.
Finally, if the dimension of the null space is $2$ or more, then there exist linearly independent (nonzero) matrices $A$ and $B$ in this null space.
If $\det{(A)}=0$, then it must have rank $1$ or $2$, and so Lemma~\ref{lem.not injective rank 2} gives that $\mathcal{A}$ is not injective.
Otherwise, consider the map
\begin{equation*}
f\colon t\mapsto \det{(A\cos{t}+B\sin{t})}\qquad\forall t\in[0,\pi].
\end{equation*}
Since $f(0)=\det{(A)}$ and $f(\pi)=\det{(-A)}=(-1)^3\det{(A)}=-\det{(A)}$, the intermediate value theorem gives that there exists $t_0\in[0,\pi]$ such that $f(t_0)=0$, i.e., the matrix $A\cos{t_0}+B\sin{t_0}$ is singular.
Moreover, this matrix is nonzero since $A$ and $B$ are linearly independent, and so its rank is either $1$ or $2$.
Lemma~\ref{lem.not injective rank 2} then gives that $\mathcal{A}$ is not injective.
\qquad
\end{proof}

As an example, we may run the HMW test on the columns of the following matrix:
\begin{equation}
\label{eq.random cahill example}
\Phi
=\left[
\begin{array}{rrrrrrrr}
2&~~1&~~1&~~0&~~0&~~0&~~1&~~\mathrm{i}\\
-1&0&0&1&1&-1&-2&2\\
0&1&-1&1&-1&2\mathrm{i}&\mathrm{i}&-1
\end{array}
\right].
\end{equation}
In this case, the null space of $\mathbf{A}$ is $1$-dimensional and spanned by a nonsingular matrix.
As such, $\mathcal{A}$ is injective. 
We will see that the HMW test has a few important applications.
First, we use it to prove the $4M-4$ Conjecture in the $M=3$ case:

\begin{theorem}
\label{thm.4M4 true 3}
The $4M-4$ Conjecture is true when $M=3$.
\end{theorem}

\begin{proof}
(a)
Suppose $N<4M-4=8$.
Then by the rank-nullity theorem, the super analysis operator $\mathbf{A}\colon\mathbb{H}^{3\times 3}\rightarrow\mathbb{R}^N$ has a null space of at least $2$ dimensions, and so by the HMW test, $\mathcal{A}$ is not injective.

(b)
Consider a $3\times 8$ matrix of real variables $\Phi(x)$ similar to $\eqref{eq.variable matrix}$.
Then $\mathcal{A}$ is injective whenever $x\in\mathbb{R}^{48}$ produces an ensemble $\{\varphi_n(x)\}_{n=1}^8\subseteq\mathbb{C}^3$ that passes the HMW test.
To pass, the rank-nullity theorem says that the null space of the super analysis operator had better be $1$-dimensional and spanned by a nonsingular matrix.
Let's use an orthonormal basis for $\mathbb{H}^{3\times 3}$ similar to \eqref{eq.onb for h} to find an $8\times 9$ matrix representation of the super analysis operator $\mathbf{A}(x)$; it is easy to check that the entries of this matrix (call it $\mathbf{A}(x)$) are polynomial functions of $x$.
Consider the matrix
\begin{equation*}
B(x,y)=
\left[
\begin{array}{c}
y^\mathrm{T}\\
\mathbf{A}(x)
\end{array}
\right],
\end{equation*}
and let $u(x)$ denote the vector of $(1,j)$th cofactors of $B(x,y)$.
Then $\langle y,u(x)\rangle=\det(B(x,y))$.
This implies that $u(x)$ is in the null space of $\mathbf{A}(x)$, since each row of $\mathbf{A}(x)$ is necessarily orthogonal to $u(x)$.

We claim that $u(x)=0$ if and only if the dimension of the null space of $\mathbf{A}(x)$ is $2$ or more, that is, the rows of $\mathbf{A}(x)$ are linearly dependent.
First, ($\Leftarrow$) is true since the entries of $u(x)$ are signed determinants of $8\times 8$ submatrices of $\mathbf{A}(x)$, which are necessarily zero by the linear dependence of the rows. 
For ($\Rightarrow$), we have that $0=\langle y,0\rangle=\langle y,u(x)\rangle=\det(B(x,y))$ for all $y\in\mathbb{R}^9$.
That is, even if $y$ is nonzero and orthogonal to the rows of $\mathbf{A}(x)$, the rows of $B(x,y)$ are linearly dependent, and so the rows of $\mathbf{A}(x)$ must be linearly dependent.
This proves our intermediate claim.

We now use the claim to prove the result.
The entries of $u(x)$ are coordinates of a matrix $U(x)\in\mathbb{H}^{3\times 3}$ in the same basis as before.
Note that the entries of $U(x)$ are polynomials of $x$.
Furthermore, $\mathcal{A}$ is injective if and only if $\det U(x)\neq0$.
To see this, observe three cases:

Case I: $U(x)=0$, i.e., $u(x)=0$, or equivalently, $\dim\operatorname{null}(\mathbf{A}(x))\geq2$.
By the HMW test, $\mathcal{A}$ is not injective.

Case II: The null space is spanned by $U(x)\neq0$, but $\det U(x)=0$.
By the HMW test, $\mathcal{A}$ is not injective.

Case III: The null space is spanned by $U(x)\neq0$, and $\det U(x)\neq0$.
By the HMW test, $\mathcal{A}$ is injective.

Defining the real algebraic variety $V=\{x:\det U(x)=0\}\subseteq\mathbb{R}^{48}$, we then have that $\mathcal{A}$ is injective precisely when $x\in V^\mathrm{c}$.
Since $V^\mathrm{c}$ is Zariski-open, it is either empty or dense with full measure, but it is nonempty since \eqref{eq.random cahill example} passes the HMW test.
Therefore, $V^\mathrm{c}$ is dense with full measure.
\qquad
\end{proof}

Recall Wright's conjecture: that there exist unitary matrices $U_1$, $U_2$ and $U_3$ such that $\Phi=[U_1~U_2~U_3]$ yields injective intensity measurements.
Also recall that Wright's conjecture implies $N^*(M)\leq3M-2$.
Again, both of these were disproved by Heinosaari et al.~\cite{HeinosaariMW:11} using deep results in differential geometry.
Alternatively, Theorem~\ref{thm.4M4 true 3} also disproves these in the case where $M=3$, since $N^*(3)=4(3)-3=8>7=3(3)-2$.

Note that the HMW test can be used to test for injectivity in three dimensions regardless of the number of measurement vectors.
As such, it can be used to evaluate ensembles of $3\times 3$ unitary matrices for quantum mechanics.
For example, consider the $3\times 3$ fractional discrete Fourier transform, defined in~\cite{CandanKO:00} using discrete Hermite-Gaussian functions:
\begin{equation*}
F^{\alpha}
=
\frac{1}{6}
\left[
\begin{array}{ccc}
3+\sqrt{3} & \sqrt{3} & \sqrt{3} \\
\sqrt{3} & \frac{3-\sqrt{3}}{2} & \frac{3-\sqrt{3}}{2} \\
\sqrt{3} & \frac{3-\sqrt{3}}{2} & \frac{3-\sqrt{3}}{2}
\end{array}
\right]
+
\frac{e^{\alpha \mathrm{i}\pi}}{6}
\left[
\begin{array}{ccc}
3-\sqrt{3} & -\sqrt{3} & -\sqrt{3} \\
-\sqrt{3} & \frac{3+\sqrt{3}}{2} & \frac{3+\sqrt{3}}{2} \\
-\sqrt{3} & \frac{3+\sqrt{3}}{2} & \frac{3+\sqrt{3}}{2}
\end{array}
\right]
+
\frac{e^{\alpha \mathrm{i}\pi/2}}{2}
\left[
\begin{array}{rrr}
0 & 0 & 0 \\
0 & 1 & -1\\
0 & -1 & 1
\end{array}
\right]
\end{equation*}
It can be shown by the HMW test that $\Phi=[I~F^{1/2}~F~F^{3/2}]$ yields injective intensity measurements.
This leads to the following refinement of Wright's conjecture:

\begin{conjecture}
\label{conjecture.new wrights}
Let $F$ denote the $M\times M$ discrete fractional Fourier transform defined in~\cite{CandanKO:00}.
Then for every $M\geq3$, $\Phi=[I~F^{1/2}~F~F^{3/2}]$ yields injective intensity measurements.
\end{conjecture}

This conjecture can be viewed as the discrete analog to the work of Jaming~\cite{Jaming:10}, in which ensembles of \textit{continuous} fractional Fourier transforms are evaluated for injectivity.

\section{Stability}

\subsection{Stability in the worst case}

As far as applications are concerned, the stability of reconstruction is perhaps the most important consideration.
To date, the only known stability results come from PhaseLift~\cite{CandesSV:11}, the polarization method~\cite{AlexeevBFM:12}, and a very recent paper of Eldar and Mendelson~\cite{EldarM:12}.
This last paper focuses on the real case, and analyzes how well subgaussian random measurement vectors distinguish signals, thereby yielding some notion of stability which is independent of the reconstruction algorithm used.
In particular, given independent random measurement vectors $\{\varphi_n\}_{n=1}^N\subseteq\mathbb{R}^M$, Eldar and Mendelson evaluated measurement separation by finding a constant $C$ such that
\begin{equation}
\label{eq.eldar stability}
\|\mathcal{A}(x)-\mathcal{A}(y)\|_1\geq C\|x-y\|_2\|x+y\|_2
\qquad
\forall x,y\in\mathbb{R}^M,
\end{equation}
where $\mathcal{A}\colon\mathbb{R}^M\rightarrow\mathbb{R}^N$ is the intensity measurement process defined by $(\mathcal{A}(x))(n):=|\langle x,\varphi_n\rangle|^2$.
With this, we can say that if $\mathcal{A}(x)$ and $\mathcal{A}(y)$ are close, then $x$ must be close to either $\pm y$, and even closer for larger $C$. 
By the contrapositive, distant signals will not be confused in the measurement domain because $\mathcal{A}$ does a good job of separating them.

One interesting feature of \eqref{eq.eldar stability} is that increasing the lengths of the measurement vectors $\{\varphi_n\}_{n=1}^N$ will in turn increase $C$, meaning the measurements are better separated.
As such, for any given magnitude of noise, one can simply amplify the measurement process so as to drown out the noise and ensure stability.
However, such amplification could be rather expensive, and so this motivates a different notion of stability---one that is invariant to how the measurement ensemble is scaled.
One approach is to build on intuition from Lemma~\ref{lem.not injective rank 2}; that is, a super analysis operator is intuitively more stable if its null space is distant from all rank-$2$ operators simultaneously; since this null space is invariant to how the measurement vectors are scaled, this is one prospective (and particularly geometric) notion of stability.
In this section, we will focus on another alternative.
Note that $\operatorname{d}(x,y):=\min\{\|x-y\|,\|x+y\|\}$ defines a metric on $\mathbb{R}^M/\{\pm1\}$, and consider the following:

\begin{definition}
\label{definition.worst-case stability}
We say $f\colon\mathbb{R}^M/\{\pm1\}\rightarrow\mathbb{R}^N$ is \textit{$C$-stable} if for every $\mathrm{SNR}>0$, there exists an estimator $g\colon\mathbb{R}^N\rightarrow\mathbb{R}^M/\{\pm1\}$ such that for every nonzero signal $x\in\mathbb{R}^M/\{\pm1\}$ and adversarial noise term $z$ with $\|z\|^2\leq\|f(x)\|^2/\mathrm{SNR}$, the relative error in reconstruction satisfies
\begin{equation*}
\frac{\operatorname{d}\big(g(f(x)+z),x\big)}{\|x\|}
\leq\frac{C}{\sqrt{\mathrm{SNR}}}.
\end{equation*}
\end{definition}

According to this definition, $f$ is more stable when $C$ is smaller.
Also, because of the $\mathrm{SNR}$ (signal-to-noise ratio) model, $f$ is $C$-stable if and only if every nonzero multiple of $f$ is also $C$-stable.
Indeed, taking $\tilde{f}:=cf$ for some nonzero scalar $c$, then for every adversarial noise term $\tilde{z}$ which is admissible for $\tilde{f}(x)$ and $\mathrm{SNR}$, we have that $z:=\tilde{z}/c$ is admissible for $f(x)$ and $\mathrm{SNR}$; as such, $\tilde{f}$ inherits $f$'s $C$-stability by using the estimator $\tilde{g}$ defined by $\tilde{g}(y):=g(y/c)$.
Overall, this notion of stability offers the invariance to scaling we originally desired.
With this, if we find a measurement process $f$ which is $C$-stable with minimal $C$, at that point, we can take advantage of noise with bounded magnitude by amplifying $f$ (and thereby effectively increasing $\mathrm{SNR}$) until the relative error in reconstruction is tolerable.

Now that we have a notion of stability, we provide a sufficient condition:

\begin{theorem}
\label{thm.bilipschitz}
Suppose $f$ is bilipschitz, that is, there exist constants $0<\alpha\leq\beta<\infty$ such that
\begin{equation*}
\alpha\operatorname{d}(x,y)
\leq\|f(x)-f(y)\|
\leq\beta\operatorname{d}(x,y)
\qquad
\forall x,y\in\mathbb{R}^M/\{\pm1\}.
\end{equation*}
If $f(0)=0$, then $f$ is $\frac{2\beta}{\alpha}$-stable.
\end{theorem}

\begin{proof}
Consider the projection function $P\colon\mathbb{R}^N\rightarrow\mathbb{R}^N$ defined by
\begin{equation*}
P(y)
:=\underset{y'\in\operatorname{range}(f)}{\arg\min}\|y'-y\|
\qquad
\forall
y\in\mathbb{R}^N.
\end{equation*}
In cases where the minimizer is not unique, we will pick one of them to be $P(y)$.
For $P$ to be well-defined, we claim it suffices for $\operatorname{range}(f)$ to be closed.
Indeed, this ensures that a minimizer always exists; since $0\in\operatorname{range}(f)$, any prospective minimizer must be no farther from $y$ than $0$ is, meaning we can equivalently minimize over the intersection of $\operatorname{range}(f)$ and the closed ball of radius $\|y\|$ centered at $y$; this intersection is compact, and so a minimizer necessarily exists.
In order to avoid using the axiom of choice, we also want a systematic method of breaking ties when the minimizer is not unique, but this can be done using lexicographic ideas provided $\operatorname{range}(f)$ is closed.

We now show that $\operatorname{range}(f)$ is, in fact, closed.
Pick a convergent sequence $\{y_n\}_{n=1}^\infty\subseteq\operatorname{range}(f)$.
This sequence is necessarily Cauchy, which means the corresponding sequence of inverse images $\{x_n\}_{n=1}^\infty\subseteq\mathbb{R}^M/\{\pm1\}$ is also Cauchy (using the lower Lipschitz bound $\alpha>0$).
Arbitrarily pick a representative $z_n\in\mathbb{R}^M$ for each $x_n$.
Then $\{z_n\}_{n=1}^\infty$ is bounded, and thus has a subsequence that converges to some $z\in\mathbb{R}^M$.
Denote $x:=\{\pm z\}\in\mathbb{R}^M/\{\pm1\}$.
Then $\operatorname{d}(x_n,x)\leq\|z_n-z\|$, and so $\{x_n\}_{n=1}^\infty$ has a subsequence which converges to $x$.
Since $\{x_n\}_{n=1}^\infty$ is also Cauchy, we therefore have $x_n\rightarrow x$.
Then the upper Lipschitz bound $\beta<\infty$ gives that $f(x)\in\operatorname{range}(f)$ is the limit of $\{y_n\}_{n=1}^\infty$.

Now that we know $P$ is well-defined, we continue.
Since $\alpha>0$, we know $f$ is injective, and so we can take $g:=f^{-1}\circ P$.
In fact, $\alpha^{-1}$ is a Lipschitz bound for $f^{-1}$, implying
\begin{equation}
\label{eq.bilipschitz1}
\operatorname{d}\big(g(f(x)+z),x\big)
=\operatorname{d}\Big(f^{-1}\big(P(f(x)+z)\big),f^{-1}\big(f(x)\big)\Big)
\leq\alpha^{-1}\|P(f(x)+z)-f(x)\|.
\end{equation}
Furthermore, the triangle inequality and the definition of $P$ together give
\begin{equation}
\label{eq.bilipschitz2}
\|P(f(x)+z)-f(x)\|
\leq\|P(f(x)+z)-(f(x)+z)\|+\|z\|
\leq\|f(x)-(f(x)+z)\|+\|z\|
=2\|z\|.
\end{equation}
Combining \eqref{eq.bilipschitz1} and \eqref{eq.bilipschitz2} then gives
\begin{equation*}
\frac{\operatorname{d}\big(g(f(x)+z),x\big)}{\|x\|}
\leq 2\alpha^{-1}\frac{\|z\|}{\|x\|}
\leq \frac{2\alpha^{-1}}{\sqrt{\mathrm{SNR}}}\frac{\|f(x)\|}{\|x\|}
=\frac{2\alpha^{-1}}{\sqrt{\mathrm{SNR}}}\frac{\|f(x)-f(0)\|}{\|x-0\|}
\leq\frac{2\beta/\alpha}{\sqrt{\mathrm{SNR}}},
\end{equation*}
as desired.
\qquad
\end{proof}

Note that the ``project-and-invert'' estimator we used to demonstrate stability is far from new.
For example, if the noise were modeled as Gaussian random, then project-and-invert is precisely the maximum likelihood estimator.
However, stochastic noise models warrant a much deeper analysis, since in this regime, one is often concerned with the bias and variance of estimates.
As such, we will investigate these issues in the next section.
Another example of project-and-invert is the Moore-Penrose pseudoinverse of an $N\times M$ matrix $A$ of rank $M$. 
Using the obvious reformulation of $C$-stable in this linear case, it can be shown that $C$ is the condition number of $A$, meaning $\alpha$ and $\beta$ are analogous to the smallest and largest singular values.
The extra factor of $2$ in the stability constant of Theorem~\ref{thm.bilipschitz} is an artifact of the nonlinear setting: 
For the sake of illustration, suppose $\operatorname{range}(f)$ is the unit circle and $f(x)=(-1,0)$ but $z=(1+\varepsilon,0)$; then $P(f(x)+z)=(1,0)$, which is just shy of $2\|z\|$ away from $f(x)$.
This sort of behavior is not exhibited in the linear case, in which $\operatorname{range}(f)$ is a subspace.

Having established the sufficiency of bilipschitz for stability, we now note that $\mathcal{A}$ is \textit{not} bilipschitz.
In fact, more generally, $\mathcal{A}$ fails to satisfy any H\"{o}lder condition. 
To see this, pick some nonzero measurement vector $\varphi_n$ and scalars $C>0$ and $\alpha\geq0$.
Then
\begin{align*}
\frac{\|\mathcal{A}((C+1)\varphi_n)-\mathcal{A}(\varphi_n)\|}{\operatorname{d}((C+1)\varphi_n,\varphi_n)^\alpha}
&=\frac{1}{\|C\varphi_n\|^\alpha}\bigg(\sum_{n'=1}^N\Big(|\langle (C+1)\varphi_n,\varphi_{n'}\rangle|^2-|\langle \varphi_n,\varphi_{n'}\rangle|^2\Big)^2\bigg)^{1/2}\\
&=\frac{(C+1)^2-1}{C^\alpha}\frac{\|\mathcal{A}(\varphi_n)\|}{\|\varphi_n\|^\alpha}.
\end{align*}
Furthermore, $\|\mathcal{A}(\varphi_n)\|\geq|(\mathcal{A}(\varphi_n))(n)|=\|\varphi_n\|^4>0$, while $\frac{(C+1)^2-1}{C^\alpha}$ diverges as $C\rightarrow\infty$, assuming $\alpha\leq1$; when $\alpha>1$, it also diverges as $C\rightarrow0$, but this case is not interesting for infamous reasons~\cite{Lipton:10}.

All is not lost, however.
As we will see, with this notion of stability, it happens to be more convenient to consider $\sqrt{\mathcal{A}}$, defined entrywise by $(\sqrt{\mathcal{A}}(x))(n)=|\langle x,\varphi_n\rangle|$.
Considering Theorem~\ref{thm.bilipschitz}, we are chiefly interested in the optimal constants $0<\alpha\leq\beta<\infty$ for which
\begin{equation}
\label{eq.restricted isometry}
\alpha \operatorname{d}(x,y)
\leq \|\sqrt{\mathcal{A}}(x)-\sqrt{\mathcal{A}}(y)\|
\leq \beta \operatorname{d}(x,y)
\qquad
\forall x,y\in \mathbb{R}^M/\{\pm1\}.
\end{equation}
In particular, Theorem~\ref{thm.bilipschitz} guarantees more stability when $\alpha$ and $\beta$ are closer together; this indicates that when suitably scaled, we want $\sqrt{\mathcal{A}}$ to act as a near-isometry, despite being a nonlinear function. 
The following lemma gives the upper Lipschitz constant:

\begin{lemma}
\label{lemma.beta expression}
The upper Lipschitz constant for $\sqrt{\mathcal{A}}$ is $\beta=\|\Phi^*\|_2$.
\end{lemma}

\begin{proof}
By the reverse triangle inequality, we have
\begin{equation*}
\big||a|-|b|\big|
\leq\min\big\{|a-b|,|a+b|\big\}
\qquad
\forall a,b\in\mathbb{R}.
\end{equation*}
Thus, for all $x,y\in\mathbb{R}^M/\{\pm1\}$,
\begin{align}
\nonumber
\|\sqrt{\mathcal{A}}(x)-\sqrt{\mathcal{A}}(y)\|^2
&=\sum_{n=1}^N\big||\langle x,\varphi_n\rangle|-|\langle y,\varphi_n\rangle|\big|^2\\
\nonumber
&\leq\sum_{n=1}^N\bigg(\min\Big\{|\langle x-y,\varphi_n\rangle|,|\langle x+y,\varphi_n\rangle|\Big\}\bigg)^2\\
\nonumber
&\leq\min\Big\{\|\Phi^*(x-y)\|^2,\|\Phi^*(x+y)\|^2\Big\}\\
\label{eq.upper lipschitz bound 1}
&\leq\|\Phi^*\|_2^2\big(\operatorname{d}(x,y)\big)^2.
\end{align}
Furthermore, picking a nonzero $x\in\mathbb{R}^M$ such that $\|\Phi^*x\|=\|\Phi^*\|_2\|x\|$ gives
\begin{equation*}
\|\sqrt{\mathcal{A}}(x)-\sqrt{\mathcal{A}}(0)\|
=\|\sqrt{\mathcal{A}}(x)\|
=\|\Phi^*x\|
=\|\Phi^*\|_2\|x\|
=\|\Phi^*\|_2\operatorname{d}(x,0),
\end{equation*}
thereby achieving equality in \eqref{eq.upper lipschitz bound 1}.
\qquad
\end{proof}

The lower Lipschitz bound is much more difficult to determine.
Our approach to analyzing this bound is based on the following definition:

\begin{definition}
We say an $M\times N$ matrix $\Phi$ satisfies the $\sigma$-strong complement property ($\sigma$-SCP) if 
\begin{equation*}
\max\big\{\lambda_{\mathrm{min}}(\Phi_S\Phi_S^*),\lambda_{\mathrm{min}}(\Phi_{S^\mathrm{c}}\Phi_{S^\mathrm{c}}^*)\big\}
\geq \sigma^2
\end{equation*}
for every $S\subseteq\{1,\ldots,N\}$.
\end{definition}

This is a numerical version of the complement property we discussed earlier.
It bears some resemblance to other matrix properties, namely combinatorial properties regarding the conditioning of submatrices, e.g., the restricted isometry property~\cite{Candes:08}, the Kadison-Singer problem~\cite{CasazzaFTW:06} and numerically erasure-robust frames~\cite{FickusM:12}.
We are interested in SCP because it is very related to the lower Lipschitz bound in \eqref{eq.restricted isometry}:

\begin{theorem}
\label{thm.bounding alpha}
The lower Lipschitz constant for $\sqrt{\mathcal{A}}$ satisfies
\begin{equation*}
\sigma\leq\alpha\leq\sqrt{2}\sigma,
\end{equation*}
where $\sigma$ is the largest scalar for which $\Phi$ has the $\sigma$-strong complement property.
\end{theorem}

\begin{proof}
By analogy with the proof of Theorem~\ref{thm.complement property characterization}, we start by proving the upper bound.
Pick $\varepsilon>0$ and note that $\Phi$ is not $(\sigma+\varepsilon)$-SCP.
Then there exists $S\subseteq\{1,\ldots,N\}$ such that both $\lambda_{\mathrm{min}}(\Phi_S\Phi_S^*)<(\sigma+\varepsilon)^2$ and $\lambda_{\mathrm{min}}(\Phi_{S^\mathrm{c}}\Phi_{S^\mathrm{c}}^*)<(\sigma+\varepsilon)^2$.
This implies that there exist unit (eigen) vectors $u,v\in\mathbb{R}^M$ such that $\|\Phi_S^*u\|<(\sigma+\varepsilon)\|u\|$ and $\|\Phi_{S^\mathrm{c}}^*v\|<(\sigma+\varepsilon)\|v\|$.
Taking $x:=u+v$ and $y:=u-v$ then gives
\begin{align*}
\|\sqrt{\mathcal{A}}(x)-\sqrt{\mathcal{A}}(y)\|^2
&=\sum_{n=1}^N\big||\langle u+v,\varphi_n\rangle|-|\langle u-v,\varphi_n\rangle|\big|^2\\
&=\sum_{n\in S}\big||\langle u+v,\varphi_n\rangle|-|\langle u-v,\varphi_n\rangle|\big|^2+\sum_{n\in S^\mathrm{c}}\big||\langle u+v,\varphi_n\rangle|-|\langle u-v,\varphi_n\rangle|\big|^2\\
&\leq4\sum_{n\in S}|\langle u,\varphi_n\rangle|^2+4\sum_{n\in S^\mathrm{c}}|\langle v,\varphi_n\rangle|^2,
\end{align*}
where the last step follows from the reverse triangle inequality.
Next, we apply our assumptions on $u$ and $v$:
\begin{align*}
\|\sqrt{\mathcal{A}}(x)-\sqrt{\mathcal{A}}(y)\|^2
&\leq4\big(\|\Phi_S^*u\|^2+\|\Phi_{S^\mathrm{c}}^*v\|^2\big)\\
&<4(\sigma+\varepsilon)^2\big(\|u\|^2+\|v\|^2\big)
=8(\sigma+\varepsilon)^2\min\big\{\|u\|^2,\|v\|^2\big\}
=2(\sigma+\varepsilon)^2\big(\operatorname{d}(x,y)\big)^2.
\end{align*}
Thus, $\alpha<\sqrt{2}(\sigma+\varepsilon)$ for all $\varepsilon>0$, and so $\alpha\leq\sqrt{2}\sigma$.

Next, to prove the lower bound, take $\varepsilon>0$ and pick $x,y\in\mathbb{R}^M/\{\pm1\}$ such that 
\begin{equation*}
(\alpha+\varepsilon)\operatorname{d}(x,y)
>\|\sqrt{\mathcal{A}}(x)-\sqrt{\mathcal{A}}(y)\|.
\end{equation*}
We will show that $\Phi$ is not $(\alpha+\varepsilon)$-SCP.
To this end, pick $S:=\{n:\operatorname{sign}\langle x,\varphi_n\rangle=-\operatorname{sign}\langle y,\varphi_n\rangle\}$ and define $u:=x+y$ and $v:=x-y$.
Then the definition of $S$ gives
\begin{equation*}
\|\Phi_S^*u\|^2
=\sum_{n\in S}|\langle x,\varphi_n\rangle+\langle y,\varphi_n\rangle|^2
=\sum_{n\in S}\big||\langle x,\varphi_n\rangle|-|\langle y,\varphi_n\rangle|\big|^2,
\end{equation*}
and similarly $\|\Phi_{S^\mathrm{c}}^*v\|^2=\sum_{n\in S^\mathrm{c}}\big||\langle x,\varphi_n\rangle|-|\langle y,\varphi_n\rangle|\big|^2$.
Adding these together then gives
\begin{equation*}
\|\Phi_S^*u\|^2+\|\Phi_{S^\mathrm{c}}^*v\|^2
=\sum_{n=1}^N\big||\langle x,\varphi_n\rangle|-|\langle y,\varphi_n\rangle|\big|^2
=\|\sqrt{\mathcal{A}}(x)-\sqrt{\mathcal{A}}(y)\|^2
<(\alpha+\varepsilon)^2\big(\operatorname{d}(x,y)\big)^2,
\end{equation*}
implying both $\|\Phi_S^*u\|<(\alpha+\varepsilon)\|u\|$ and $\|\Phi_{S^\mathrm{c}}^*v\|<(\alpha+\varepsilon)\|v\|$.
Therefore, $\Phi$ is not $(\alpha+\varepsilon)$-SCP, i.e., $\sigma<\alpha+\varepsilon$ for all $\varepsilon>0$, which in turn implies the desired lower bound.
\qquad
\end{proof}

Note that all of this analysis specifically treats the real case; indeed, the metric we use would not be appropriate in the complex case.
However, just like the complement property is necessary for injectivity in the complex case (Theorem~\ref{thm.complex complement property}), we suspect that the strong complement property is necessary for stability in the complex case, but we have no proof of this.

As an example of how to apply Theorem~\ref{thm.bounding alpha}, pick $M$ and $N$ to both be even and let $F=\{f_n\}_{n\in\mathbb{Z}_N}$ be the $\frac{M}{2}\times N$ matrix you get by collecting the first $\frac{M}{2}$ rows of the $N\times N$ discrete Fourier transform matrix with entries of unit modulus.
Next, take $\Phi=\{\varphi_n\}_{n\in\mathbb{Z}_N}$ to be the $M\times N$ matrix you get by stacking the real and imaginary parts of $F$ and normalizing the resulting columns (i.e., multiplying by $\sqrt{2/M}$).
Then $\Phi$ happens to be a \textit{self-localized finite frame} due to the rapid decay in coherence between columns.
To be explicit, first note that 
\begin{align*}
|\langle \varphi_n,\varphi_{n'}\rangle|^2
&=\tfrac{4}{M^2}|\langle \operatorname{Re}f_n,\operatorname{Re}f_{n'}\rangle+\langle \operatorname{Im}f_n,\operatorname{Im}f_{n'}\rangle|^2\\
&\leq\tfrac{4}{M^2}\Big|\Big(\langle \operatorname{Re}f_n,\operatorname{Re}f_{n'}\rangle+\langle \operatorname{Im}f_n,\operatorname{Im}f_{n'}\rangle\Big)+i\Big(\langle \operatorname{Im}f_n,\operatorname{Re}f_{n'}\rangle-\langle \operatorname{Re}f_n,\operatorname{Im}f_{n'}\rangle\Big)\Big|^2\\
&=\tfrac{4}{M^2}|\langle f_n,f_{n'}\rangle|^2,
\end{align*}
and furthermore, when $n\neq n'$, the geometric sum formula gives
\begin{equation*}
|\langle f_n,f_{n'}\rangle|^2
=\bigg|\sum_{m=0}^{M-1}e^{2\pi im(n-n')/N}\bigg|^2
=\frac{\sin^2(M\pi(n-n')/N)}{\sin^2(\pi(n-n')/N)}
\leq\frac{1}{\sin^2(\pi(n-n')/N)}.
\end{equation*}
Taking $u:=\varphi_0$, $v:=\varphi_{N/2}$ and $S:=\{n:\frac{N}{4}\leq n<\frac{3N}{4}\}$, we then have
\begin{equation*}
\frac{\|\Phi_S^*u\|^2}{\|u\|^2}
=\|\Phi_S^*u\|^2
=\sum_{n\in S}|\langle\varphi_0,\varphi_n\rangle|^2
\leq\frac{4}{M^2}\sum_{n\in S}\frac{1}{\sin^2(\pi n/N)}
\leq\frac{4}{M^2}\cdot\frac{N/2}{\sin^2(\pi/4)}
=\frac{4N}{M^2},
\end{equation*}
and similarly for $\frac{\|\Phi_{S^\mathrm{c}}^*v\|^2}{\|v\|^2}$.
As such, if $N=o(M^2)$, then $\Phi$ is $\sigma$-SCP only if $\sigma$ vanishes, meaning phase retrieval with $\Phi$ necessarily lacks the stability guarantee of Theorem~\ref{thm.bounding alpha}.
As a rule of thumb, self-localized frames fail to provide stable phase retrieval for this very reason; just as we cannot stably distinguish between $\varphi_0+\varphi_{N/2}$ and $\varphi_0-\varphi_{N/2}$ in this case, in general, signals consisting of ``distant'' components bring similar instability.
This intuition was first pointed out to us by Irene Waldspurger---we simply made it more rigorous with the notion of SCP.
This means that stable phase retrieval from localized measurements must either use prior information about the signal (e.g., connected support) or additional measurements; indeed, this dichotomy has already made its mark on the Fourier-based phase retrieval literature~\cite{Fannjiang:12,JaganathanOH:12}.

We can also apply the strong complement property to show that certain (random) ensembles produce stable measurements.
We will use the following lemma, which is proved in the proof of Lemma~4.1 in~\cite{ChenD:05}:

\begin{lemma}
\label{lemma.random bound smallest eigenvalue}
Given $n\geq m\geq2$, draw a real $m\times n$ matrix $G$ of independent standard normal entries.
Then
\begin{equation*}
\operatorname{Pr}\bigg(\lambda_\mathrm{min}(GG^*)\leq\frac{n}{t^2}\bigg)\leq\frac{1}{\Gamma(n-m+2)}\bigg(\frac{n}{t}\bigg)^{n-m+1}
\qquad
\forall t>0.
\end{equation*}
\end{lemma}

\begin{theorem}
\label{theorem.random scp bound}
Draw an $M\times N$ matrix $\Phi$ with independent standard normal entries, and denote $R=\frac{N}{M}$.
Provided $R>2$, then for every $\varepsilon>0$, $\Phi$ has the $\sigma$-strong complement property with
\begin{equation*}
\sigma
=\frac{1}{\sqrt{2}e^{1+\varepsilon/(R-2)}}\cdot\frac{N-2M+2}{2^{R/(R-2)}\sqrt{N}},
\end{equation*}
with probability $\geq1-e^{-\varepsilon M}$.
\end{theorem}

\begin{proof}
Fix $M$ and $N$, and consider the function $f\colon(M-2,\infty)\rightarrow(0,\infty)$ defined by
\begin{equation*}
f(x):=\frac{1}{\Gamma(x-M+2)}(\sigma\sqrt{x})^{x-M+1}.
\end{equation*}
To simplify our analysis, we will assume that $N$ is even, but the proof can be amended to account for the odd case.
Applying Lemma~\ref{lemma.random bound smallest eigenvalue}, we have for every subset $S\subseteq\{1,\ldots,N\}$ of size $K$ that $\operatorname{Pr}(\lambda_\mathrm{min}(\Phi_S\Phi_S^*)<\sigma^2)\leq f(K)$, provided $K\geq M$, and similarly $\operatorname{Pr}(\lambda_\mathrm{min}(\Phi_{S^\mathrm{c}}\Phi_{S^\mathrm{c}}^*)<\sigma^2)\leq f(N-K)$, provided $N-K\geq M$.
We will use this to bound the probability that $\Phi$ is not $\sigma$-SCP.
Since $\lambda_\mathrm{min}(\Phi_{S^\mathrm{c}}\Phi_{S^\mathrm{c}}^*)=0$ whenever $|S|\geq N-M+1$ and $\lambda_\mathrm{min}(\Phi_S\Phi_S^*)\leq\lambda_\mathrm{min}(\Phi_T\Phi_T^*)$ whenever $S\subseteq T$, then a union bound gives
\begin{align}
\nonumber
&\operatorname{Pr}\Big(\Phi\mbox{ is not $\sigma$-SCP}\Big)\\
\nonumber
&=\operatorname{Pr}\Big(\exists S\subseteq\{1,\ldots,N\}\mbox{ s.t. }\lambda_\mathrm{min}(\Phi_S\Phi_S^*)<\sigma^2\mbox{ and }\lambda_\mathrm{min}(\Phi_{S^\mathrm{c}}\Phi_{S^\mathrm{c}}^*)<\sigma^2\Big)\\
\nonumber
&\leq\operatorname{Pr}\Big(\exists S\subseteq\{1,\ldots,N\},|S|=N-M+1,\mbox{ s.t. }\lambda_\mathrm{min}(\Phi_S\Phi_S^*)<\sigma^2\Big)\\
\nonumber
&\qquad+\operatorname{Pr}\Big(\exists S\subseteq\{1,\ldots,N\},M\leq|S|\leq N-M,\mbox{ s.t. }\lambda_\mathrm{min}(\Phi_S\Phi_S^*)<\sigma^2\mbox{ and }\lambda_\mathrm{min}(\Phi_{S^\mathrm{c}}\Phi_{S^\mathrm{c}}^*)<\sigma^2\Big)\\
\label{eq.f expression to bound}
&\leq\binom{N}{N-M+1}f(N-M+1)+\frac{1}{2}\sum_{K=M}^{N-M}\binom{N}{K}f(K)f(N-K),
\end{align}
where the last inequality follows in part from the fact that $\lambda_\mathrm{min}(\Phi_S\Phi_S^*)$ and $\lambda_\mathrm{min}(\Phi_{S^\mathrm{c}}\Phi_{S^\mathrm{c}}^*)$ are independent random variables, and the factor $\frac{1}{2}$ is an artifact of double counting partitions.
We will further bound each term in \eqref{eq.f expression to bound} to get a simpler expression.
First, $\binom{2k}{k}\geq 2^k$ for all $k$ and so 
\begin{align*}
f(N-M+1)
&\leq\frac{1}{\Gamma(N-2M+3)}(\sigma\sqrt{N})^{N-2M+2}\\
&\leq\frac{1}{\Gamma(N-2M+3)}(\sigma\sqrt{N})^{N-2M+2}\cdot\frac{1}{2^{\frac{N}{2}-M+1}}\binom{N-2M+2}{\frac{N}{2}-M+1}
=f(\tfrac{N}{2})^2
\end{align*}
Next, we will find that $g(x):=f(x)f(N-x)$ is maximized at $x=\frac{N}{2}$.
To do this, we first find the critical points of $g$.
Since $0=g'(x)=f'(x)f(N-x)-f(x)f'(N-x)$, we have
\begin{equation}
\label{eq.derivative of log f}
\frac{d}{dy}\log f(y)\bigg|_{y=x}
=\frac{f'(x)}{f(x)}
=\frac{f'(N-x)}{f(N-x)}
=\frac{d}{dy}\log f(y)\bigg|_{y=N-x}.
\end{equation}
To analyze this further, we take another derivative:
\begin{equation}
\label{eq.second derivative of log f}
\frac{d^2}{dy^2}\log f(y)
=\frac{1}{2y}+\frac{M-1}{2y^2}-\frac{d^2}{dy^2}\log\Gamma(y-M+2).
\end{equation}
It is straightforward to see that
\begin{equation*}
\frac{1}{2y}+\frac{M-1}{2y^2}
\leq\frac{1}{y-M+2}
=\int_{y-M+2}^\infty\frac{dt}{t^2}
<\sum_{k=0}^\infty\frac{1}{(y-M+2+k)^2}
=\frac{d^2}{dy^2}\log\Gamma(y-M+2),
\end{equation*}
where the last step uses a series expression for the trigamma function $\psi_1(z):=\frac{d^2}{dz^2}\log\Gamma(z)$; see Section~6.4 of~\cite{AbramowitzS:64}.
Applying this to \eqref{eq.second derivative of log f} then gives that $\frac{d^2}{dy^2}\log f(y)<0$, which in turn implies that $\frac{d}{dy}\log f(y)$ is strictly decreasing in $y$.
Thus, \eqref{eq.derivative of log f} requires $x=N-x$, and so $x=\frac{N}{2}$ is the only critical point of $g$.
Furthermore, to see that this is a maximizer, notice that
\begin{equation*}
g''(\tfrac{N}{2})
=2f(\tfrac{N}{2})^2\cdot\frac{f''(\tfrac{N}{2})f(\tfrac{N}{2})-f'(\tfrac{N}{2})^2}{f(\tfrac{N}{2})^2}
=2f(\tfrac{N}{2})^2\cdot\frac{d}{dy}\frac{f'(y)}{f(y)}\bigg|_{y=\frac{N}{2}}
=2f(\tfrac{N}{2})^2\cdot\frac{d^2}{dy^2}\log f(y)\bigg|_{y=\frac{N}{2}}
<0.
\end{equation*}
To summarize, we have that $f(N-M+1)$ and $f(K)f(N-K)$ are both at most $f(\tfrac{N}{2})^2$.
This leads to the following bound on \eqref{eq.f expression to bound}:
\begin{equation*}
\operatorname{Pr}\Big(\Phi\mbox{ is not $\sigma$-SCP}\Big)
\leq\frac{1}{2}\sum_{K=0}^{N}\binom{N}{K}f(\tfrac{N}{2})^2
=2^{N-1}f(\tfrac{N}{2})^2
=\frac{2^{N-1}}{\Gamma(\frac{N}{2}-M+2)^2}\Big(\sigma\sqrt{\tfrac{N}{2}}\Big)^{N-2M+2}.
\end{equation*}
Finally, applying the fact that $\Gamma(k+1)\geq e(\frac{k}{e})^k$ gives
\begin{align*}
\operatorname{Pr}\Big(\Phi\mbox{ is not $\sigma$-SCP}\Big)
&\leq\frac{2^{N-1}}{e^2}\bigg(\sigma e\sqrt{2}\cdot\frac{\sqrt{N}}{N-2M+2}\bigg)^{N-2M+2}\\
&=\frac{2^{RM}}{2e^2}\Big(e^{-\varepsilon/(R-2)}2^{-R/(R-2)}\Big)^{(R-2)M+2}
\leq 2^{RM}(e^{\varepsilon}2^R)^{-M}
=e^{-\varepsilon M},
\end{align*}
as claimed.
\qquad
\end{proof}

Considering $\|\Phi^*\|_2\leq(1+\varepsilon)(\sqrt{N}+\sqrt{M})$ with probability $\geq1-2e^{-\varepsilon(\sqrt{N}^2+\sqrt{M})^2/2}$ (see Theorem~II.13 of~\cite{DavidsonS:01}), we can leverage Theorem~\ref{theorem.random scp bound} to determine the stability of a Gaussian measurement ensemble.
Specifically, we have by Theorem~\ref{thm.bilipschitz} along with Lemma~\ref{lemma.beta expression} and Theorem~\ref{thm.bounding alpha} that such measurements are $C$-stable with
\begin{equation}
\label{eq.bounds on random C}
C=\frac{2\beta}{\alpha}
\leq\frac{2\|\Phi^*\|_2}{\sigma}
\sim \underbrace{2(\sqrt{N}+\sqrt{M})\cdot\sqrt{2}e\cdot\frac{2^{R/(R-2)}\sqrt{N}}{N-2M+2}}_{a(R,M)}
\leq \underbrace{2\sqrt{2}e\bigg(\frac{R+\sqrt{R}}{R-2}\bigg)2^{R/(R-2)}}_{b(R)}
\end{equation}
Figure~\ref{figure.curves} illustrates these bounds along with different realizations of $2\|\Phi^*\|_2/\sigma$.
This suggests that the redundancy of the measurement process is the main factor that determines stability of a random measurement ensemble (and that bounded redundancies suffice for stability).
Furthermore, the project-and-invert estimator will yield particularly stable signal reconstruction, although it is not obvious how to efficiently implement this estimator; this is one advantage given by the reconstruction algorithms in~\cite{AlexeevBFM:12,CandesSV:11}.

\begin{figure}[t]
\label{figure.curves}
\begin{center}
\begin{tabular}{cccc}
\includegraphics[height=150pt]{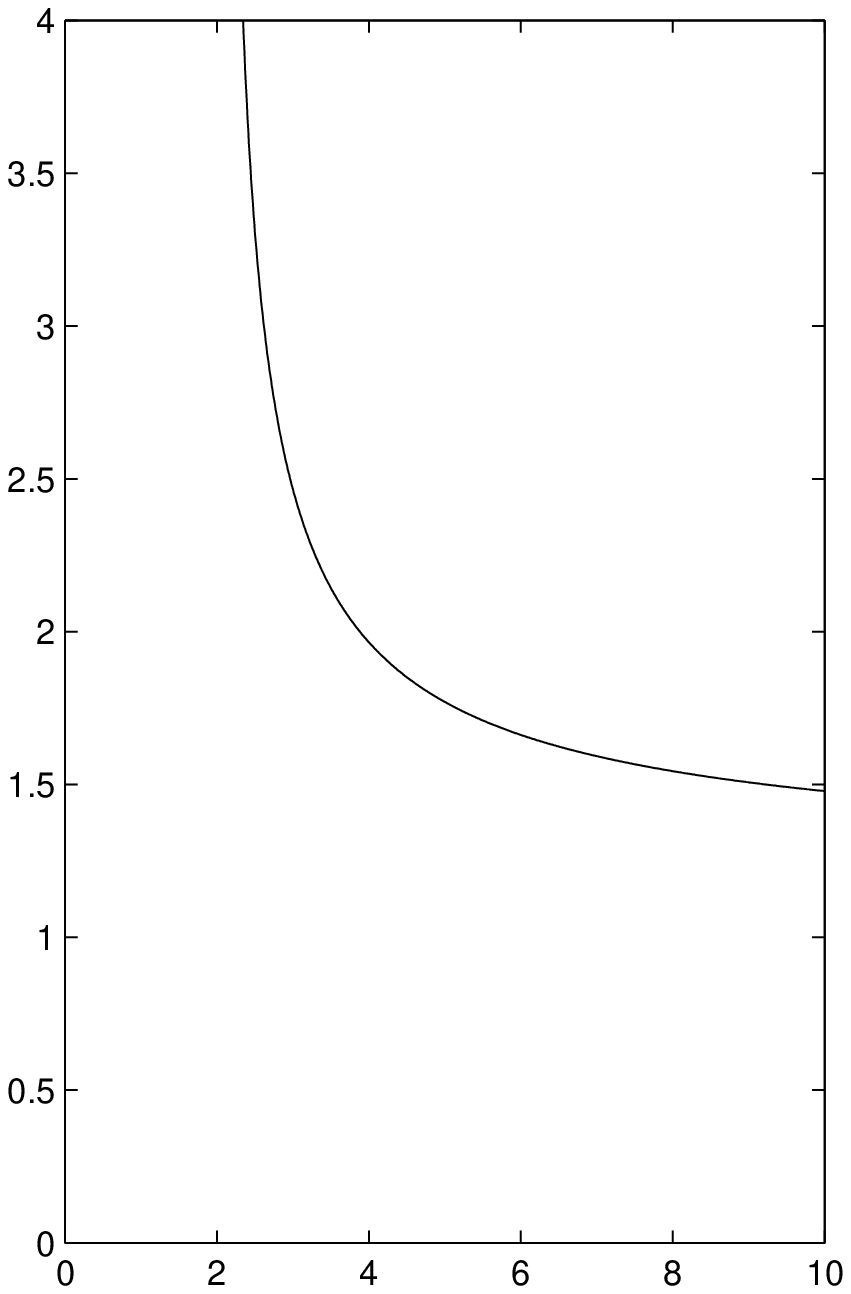}&
\includegraphics[height=150pt]{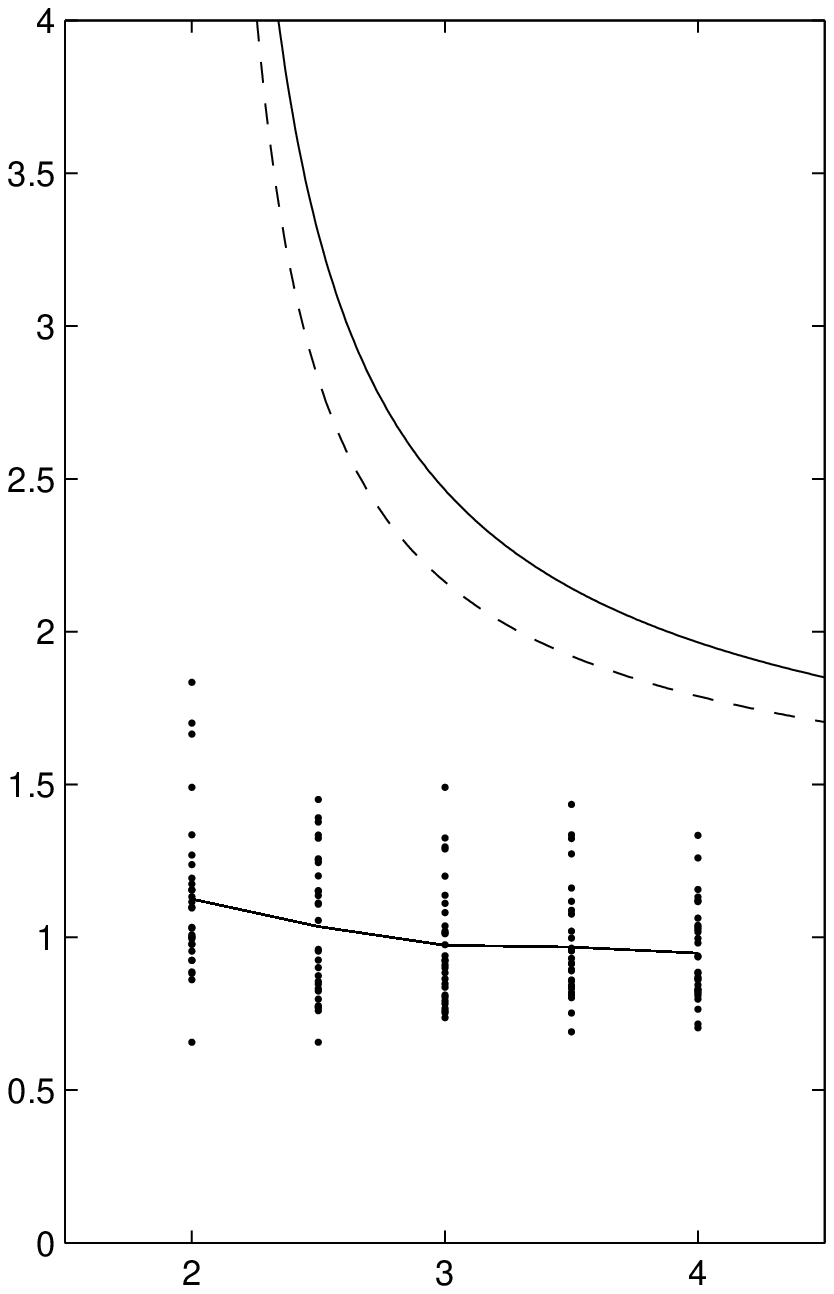}&
\includegraphics[height=150pt]{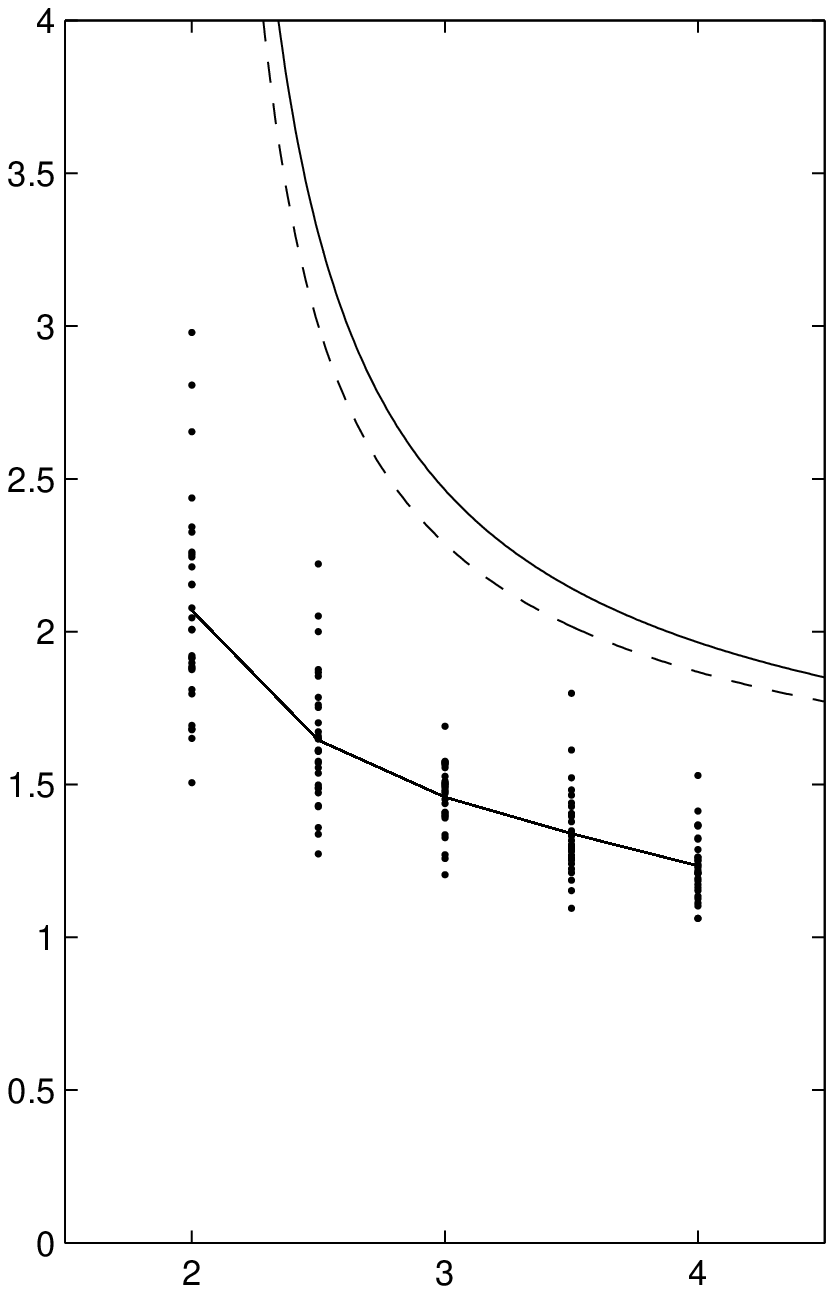}&
\includegraphics[height=150pt]{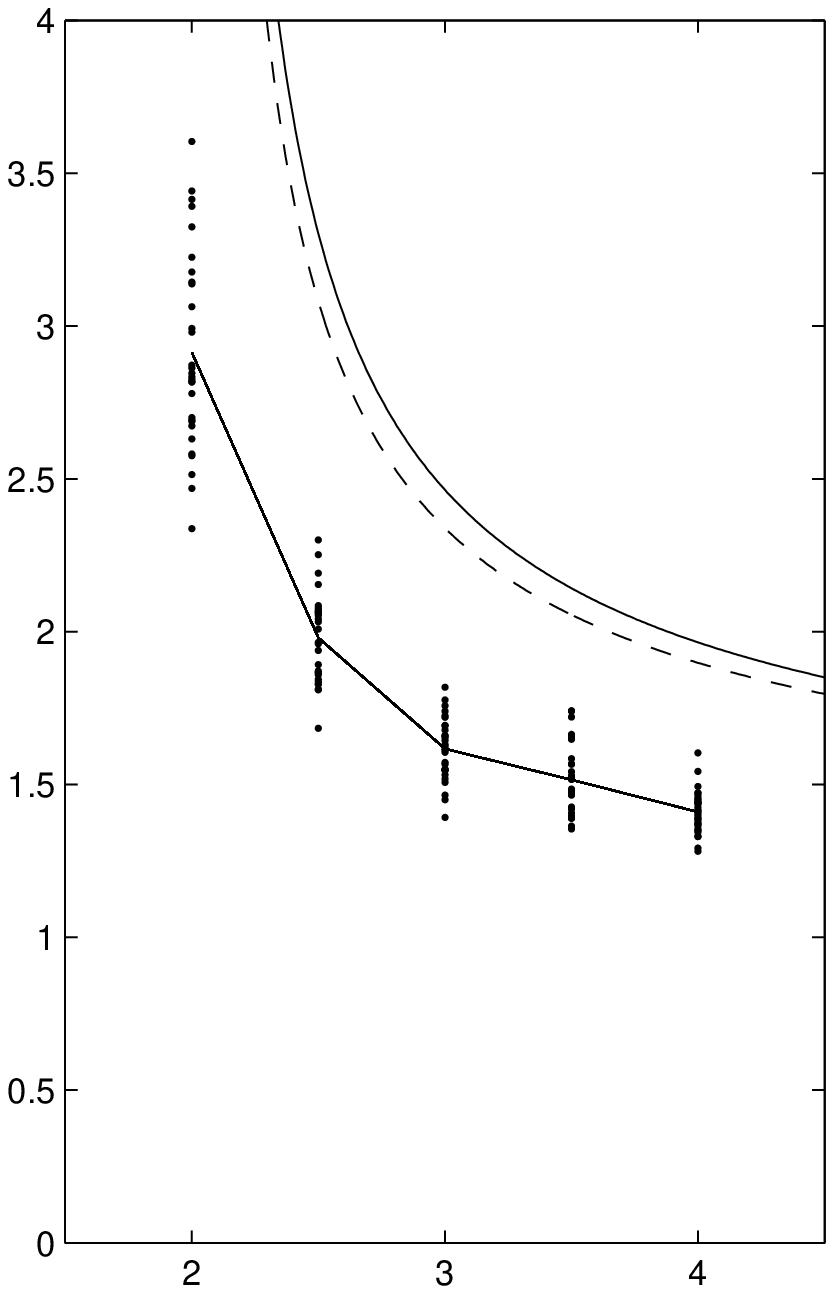}
\end{tabular}
\end{center}
\caption{
The graph on the left depicts $\log_{10} b(R)$ as a function of $R$, which is defined in \eqref{eq.bounds on random C}.
Modulo $\varepsilon$ terms, this serves as an upper bound on $\log_{10}(2\|\Phi^*\|_2/\sigma)$ with high probability as $M\rightarrow\infty$, where $\Phi$ is an $M\times RM$ matrix of independent standard Gaussian entries.
Based on Theorem~\ref{thm.bilipschitz} (along with Lemma~\ref{lemma.beta expression} and Theorem~\ref{thm.bounding alpha}), this provides a stability guarantee for the corresponding measurement process, namely $\sqrt{\mathcal{A}}$.
Since $\log_{10} b(R)$ exhibits an asymptote at $R=2$, this gives no stability guarantee for measurement ensembles of redundancy $2$.
The next three graphs consider the special cases where $M=2,4,6$, respectively.
In each case, the dashed curve depicts the slightly stronger upper bound of $\log_{10} a(R,M)$, defined in \eqref{eq.bounds on random C}.
Also depicted, for each $R\in\{2,2.5,3,3.5,4\}$, are $30$ realizations of $\log_{10}(2\|\Phi^*\|_2/\sigma)$; we provide a piecewise linear graph connecting the sample averages for clarity.
Notice that as $M$ increases, $\log_{10} a(R,M)$ approaches $\log_{10} b(R)$; this is easily seen by their definitions in \eqref{eq.bounds on random C}.
More interestingly, the random realizations also appear to be approaching $\log_{10} b(R)$; this is most notable with the realizations corresponding to $R=2$.
To be clear, we use $\sigma$ as a proxy for $\alpha$ (see Theorem~\ref{thm.bounding alpha}) because $\alpha$ is particularly difficult to obtain; as such, we do not plot realizations of $\log_{10}(2\beta/\alpha)$.
}
\end{figure}

\subsection{Stability in the average case}

Suppose a random variable $Y$ is drawn according to some unknown member of a parameterized family of probability density functions $\{f(\cdot;\theta)\}_{\theta\in\Omega}$.
The Fisher information $J(\theta)$ quantifies how much information about the unknown parameter $\theta$ is given by the random variable on average.
This is particularly useful in statistical signal processing, where a signal measurement is corrupted by random noise, and the original signal is viewed as a parameter of the random measurement's unknown probability density function; as such, the Fisher information quantifies how useful the noisy measurement is for signal estimation.

In this section, we will apply the theory of Fisher information to evaluate the stability of $\mathcal{A}$.
To do this, we consider a stochastic noise model, that is, given some signal $x$, we take measurements of the form $Y=\mathcal{A}(x)+Z$, where the entries of $Z$ are independent Gaussian random variables with mean $0$ and variance $\sigma^2$.
We want to use $Y$ to estimate $x$ up to a global phase factor; to simplify the analysis, we will estimate a particular $\theta(x)\equiv x$, specifically (and arbitrarily) $x$ divided by the phase of its last nonzero entry.
As such, $Y$ is a random vector with probability density function
\begin{equation*}
f(y;\theta)
=\frac{1}{(2\pi \sigma^2)^{N/2}}e^{-\|y-\mathcal{A}(\theta)\|^2/2\sigma^2}
\qquad
\forall y\in\mathbb{R}^N.
\end{equation*}
With this, we can calculate the Fisher information matrix, defined entrywise by
\begin{equation}
\label{eq.fisher information}
\big(J(\theta)\big)_{ij}
:=\mathbb{E}\bigg[\bigg(\frac{\partial}{\partial\theta_i}\log f(Y;\theta)\bigg)\bigg(\frac{\partial}{\partial\theta_j}\log f(Y;\theta)\bigg)\bigg|\theta\bigg].
\end{equation}
In particular, we have
\begin{equation*}
\frac{\partial}{\partial\theta_i}\log f(y;\theta)
=\frac{\partial}{\partial\theta_i}\bigg(-\frac{1}{2\sigma^2}\sum_{n=1}^N\Big(y_n-\big(\mathcal{A}(\theta)\big)_n\Big)^2\bigg)
=\frac{1}{\sigma^2}\sum_{n=1}^N\Big(y_n-\big(\mathcal{A}(\theta)\big)_n\Big)\frac{\partial}{\partial\theta_i}\big(\mathcal{A}(\theta)\big)_n,
\end{equation*}
and so applying \eqref{eq.fisher information} along with the independence of the entries of $Z$ gives
\begin{equation*}
\big(J(\theta)\big)_{ij}
=\frac{1}{\sigma^4}\sum_{n=1}^N\sum_{n'=1}^N\frac{\partial}{\partial\theta_i}\big(\mathcal{A}(\theta)\big)_n\frac{\partial}{\partial\theta_j}\big(\mathcal{A}(\theta)\big)_{n'}\mathbb{E}[Z_nZ_{n'}]
=\frac{1}{\sigma^2}\sum_{n=1}^N\frac{\partial}{\partial\theta_i}\big(\mathcal{A}(\theta)\big)_n\frac{\partial}{\partial\theta_j}\big(\mathcal{A}(\theta)\big)_{n}.
\end{equation*}
It remains to take partial derivatives of $\mathcal{A}(\theta)$, but this calculation depends on whether $\theta$ is real or complex.

In the real case, we have
\begin{equation*}
\frac{\partial}{\partial\theta_i}\big(\mathcal{A}(\theta)\big)_n
=\frac{\partial}{\partial\theta_i}\bigg(\sum_{m=1}^M\theta_m\varphi_n(m)\bigg)^2
=2\bigg(\sum_{m=1}^M\theta_m\varphi_n(m)\bigg)\varphi_n(i).
\end{equation*}
Thus, if we take $\Psi(\theta)$ to be the $M\times N$ matrix whose $n$th column is $\langle\theta,\varphi_n\rangle \varphi_n$, then the Fisher information matrix can be expressed as $J(\theta)=\frac{4}{\sigma^2}\Psi(\theta)\Psi(\theta)^*$.
Interestingly, Theorem~\ref{thm.complement property characterization} implies that $J(\theta)$ is necessarily positive definite when $\mathcal{A}$ is injective.
To see this, suppose there exists $\theta\in\Omega$ such that $J(\theta)$ has a nontrivial null space.
Then $\{\langle\theta,\varphi_n\rangle \varphi_n\}_{n=1}^N$ does not span $\mathbb{R}^M$, and so $S=\{n:\langle\theta,\varphi_n\rangle=0\}$ breaks the complement property.
As the following result shows, when $\mathcal{A}$ is injective, the conditioning of $J(\theta)$ lends some insight into stability:

\begin{theorem}
\label{thm.real random stability}
For $x\in\mathbb{R}^M$, let $Y=\mathcal{A}(x)+Z$ denote noisy intensity measurements with $Z$ having independent $\mathcal{N}(0,\sigma^2)$ entries.
Furthermore, define the parameter $\theta$ to be $x$ divided by the sign of its last nonzero entry; let $\Omega\subseteq\mathbb{R}^{M}$ denote all such $\theta$.
Then for any unbiased estimator $\hat{\theta}(Y)$ of $\theta$ in $\Omega$ with a finite $M\times M$ covariance matrix $C(\hat{\theta})$, we have $C(\hat{\theta})-J(\theta)^{-1}$ is positive semidefinite whenever $\theta\in\operatorname{int}(\Omega)$. 
\end{theorem}

This result was first given by Balan (see Theorem~4.1 in~\cite{Balan:12}).
Note that the requirement that $\theta$ be in the interior of $\Omega$ can be weakened to $\theta\neq0$ by recognizing that our choice for $\theta$ (dividing by the sign of the last nonzero entry) was arbitrary.
To interpret this theorem, note that
\begin{align*}
\operatorname{Tr}[C(\hat{\theta})]
&=\operatorname{Tr}[\mathbb{E}[(\hat{\theta}(Y)-\theta)(\hat{\theta}(Y)-\theta)^\mathrm{T}]]\\
&=\mathbb{E}[\operatorname{Tr}[(\hat{\theta}(Y)-\theta)(\hat{\theta}(Y)-\theta)^\mathrm{T}]]
=\mathbb{E}[\operatorname{Tr}[(\hat{\theta}(Y)-\theta)^\mathrm{T}(\hat{\theta}(Y)-\theta)]]
=\mathbb{E}\|\hat{\theta}(Y)-\theta\|^2,
\end{align*}
and so Theorem~\ref{thm.real random stability} and the linearity of the trace together give $\mathbb{E}\|\hat{\theta}(Y)-\theta\|^2=\operatorname{Tr}[C(\hat{\theta})]\geq\operatorname{Tr}[J(\theta)^{-1}]$.
In the previous section, Definition~\ref{definition.worst-case stability} provided a notion of worst-case stability based on the existence of an estimator with small error. 
By analogy, Theorem~\ref{thm.real random stability} demonstrates a converse of sorts: that no unbiased estimator will have mean squared error smaller than $\operatorname{Tr}[J(\theta)^{-1}]$.
As such, a stable measurement ensemble might minimize $\sup_{\theta\in\Omega}\operatorname{Tr}[J(\theta)^{-1}]$, although this is a particularly cumbersome objective function to work with.
More interestingly, Theorem~\ref{thm.real random stability} provides another numerical strengthening of the complement property (analogous to the strong complement property of the previous section).
Unfortunately, we cannot make a more rigorous comparison between the worst- and average-case analyses of stability; indeed, our worst-case analysis exploited the fact that $\sqrt{\mathcal{A}}$ is bilipschitz (which $\mathcal{A}$ is not), and as we shall see, the average-case analysis depends on $\mathcal{A}$ being differentiable (which $\sqrt{\mathcal{A}}$ is not).

To calculate the information matrix in the complex case, we first express our parameter vector in real coordinates: $\theta=(\theta_1+\mathrm{i}\theta_{M+1},\theta_2+\mathrm{i}\theta_{M+2},\ldots,\theta_M+\mathrm{i}\theta_{2M})$, that is, we view $\theta$ as a $2M$-dimensional real vector by concatenating its real and imaginary parts.
Next, for any arbitrary function $g\colon\mathbb{R}^{2M}\rightarrow\mathbb{C}$, the product rule gives
\begin{equation}
\label{eq.fisher complex 1}
\frac{\partial}{\partial\theta_i}|g(\theta)|^2
=\frac{\partial}{\partial\theta_i}g(\theta)\overline{g(\theta)}
=\bigg(\frac{\partial}{\partial\theta_i}g(\theta)\bigg)\overline{g(\theta)}+g(\theta)\overline{\bigg(\frac{\partial}{\partial\theta_i}g(\theta)\bigg)}
=2\operatorname{Re}g(\theta)\frac{\partial}{\partial\theta_i}\overline{g(\theta)}.
\end{equation}
Since we care about partial derivatives of $\mathcal{A}(\theta)$, we take $g(\theta)=\langle\theta,\varphi_n\rangle=\sum_{m=1}^M(\theta_m+\mathrm{i}\theta_{M+m})\overline{\varphi_n(m)}$, and so 
\begin{equation}
\label{eq.fisher complex 2}
\frac{\partial}{\partial\theta_i}\overline{g(\theta)}
=\left\{
\begin{array}{cl}
\varphi_n(i) &\mbox{if } i\leq M\\
-\mathrm{i}\varphi_n(i-M) &\mbox{if } i>M.
\end{array}
\right.
\end{equation}
Combining \eqref{eq.fisher complex 1} and \eqref{eq.fisher complex 2} then gives the following expression for the Fisher information matrix:
Take $\Psi(\theta)$ to be the $2M\times N$ matrix whose $n$th column is formed by stacking the real and imaginary parts of $\langle\theta,\varphi_n\rangle \varphi_n$; then $J(\theta)=\frac{4}{\sigma^2}\Psi(\theta)\Psi(\theta)^*$.

\begin{lemma}
\label{lemma.J is invertible}
Take $\widetilde{J}(\theta)$ to be the $(2M-1)\times(2M-1)$ matrix that comes from removing the last row and column of $J(\theta)$.
If $\mathcal{A}$ is injective, then $\widetilde{J}(\theta)$ is positive definite for every $\theta\in\operatorname{int}(\Omega)$.
\end{lemma}

\begin{proof}
First, we note that $J(\theta)=\frac{4}{\sigma^2}\Psi(\theta)\Psi(\theta)^*$ is necessarily positive semidefinite, and so 
\begin{equation*}
\inf_{\|x\|=1}x^\mathrm{T}\tilde{J}(\theta)x
=\inf_{\|x\|=1}[x; 0]^\mathrm{T}J(\theta)[x; 0]
\geq\inf_{\|y\|=1}y^\mathrm{T}J(\theta)y
\geq0.
\end{equation*}
As such, it suffices to show that $\tilde{J}(\theta)$ is invertible.

To this end, take any vector $x$ in the null space of $\widetilde{J}(\theta)$.
Then defining $y:=[x; 0]\in\mathbb{R}^{2M}$, we have that $J(\theta)y$ is zero in all but (possibly) the $2M$th entry.
As such, $0=\langle y,J(\theta)y\rangle=\|\frac{2}{\sigma}\Psi(\theta)^*y\|^2$, meaning $y$ is orthogonal to the columns of $\Psi(\theta)$.
Since $\mathcal{A}$ is injective, Theorem~\ref{thm.complex injective} then gives that $y=\alpha\mathrm{i}\theta$ for some $\alpha\in\mathbb{R}$.
But since $\theta\in\operatorname{int}(\Omega)$, we have $\theta_M>0$, and so the $2M$th entry of $\mathrm{i}\theta$ is necessarily nonzero.
This means $\alpha=0$, and so $y$ (and thus $x$) is trivial.
\qquad
\end{proof}

\begin{theorem}
\label{thm.complex random stability}
For $x\in\mathbb{C}^M$, let $Y=\mathcal{A}(x)+Z$ denote noisy intensity measurements with $Z$ having independent $\mathcal{N}(0,\sigma^2)$ entries.
Furthermore, define the parameter $\theta$ to be $x$ divided by the phase of its last nonzero entry, and view $\theta$ as a vector in $\mathbb{R}^{2M}$ by concatenating its real and imaginary parts; let $\Omega\subseteq\mathbb{R}^{2M}$ denote all such $\theta$.
Then for any unbiased estimator $\hat{\theta}(Y)$ of $\theta$ in $\Omega$ with a finite $2M\times 2M$ covariance matrix $C(\hat{\theta})$, the last row and column of $C(\hat{\theta})$ are both zero, and the remaining $(2M-1)\times(2M-1)$ submatrix $\widetilde{C}(\hat{\theta})$ has the property that $\widetilde{C}(\hat{\theta})-\widetilde{J}(\theta)^{-1}$ is positive semidefinite whenever $\theta\in\operatorname{int}(\Omega)$. 
\end{theorem}

\begin{proof}
We start by following the usual proof of the vector parameter Cramer-Rao lower bound (see for example Appendix 3B of~\cite{Kay:93}).
Note that for any $i,j\in\{1,\ldots,2M\}$,
\begin{align*}
\int_{\mathbb{R}^N}\big((\hat{\theta}(y))_j-\theta_j\big)\frac{\partial\log f(y;\theta)}{\partial\theta_i}f(y;\theta)dy
&=\int_{\mathbb{R}^N}(\hat{\theta}(y))_j\frac{\partial f(y;\theta)}{\partial\theta_i}dy-\theta_j\int_{\mathbb{R}^N}\frac{\partial f(y;\theta)}{\partial\theta_i}dy\\
&=\frac{\partial}{\partial\theta_i}\int_{\mathbb{R}^N}(\hat{\theta}(y))_jf(y;\theta)dy-\theta_j\frac{\partial}{\partial\theta_i}\int_{\mathbb{R}^N}f(y;\theta)dy,
\end{align*}
where the second equality is by differentiation under the integral sign (see Lemma~\ref{lemma.differentiation under the integral sign} for details; here, we use the fact that $\hat{\theta}$ has a finite covariance matrix so that $\hat{\theta}_j$ has a finite second moment).
Next, we use the facts that $\hat{\theta}$ is unbiased and $f(\cdot;\theta)$ is a probability density function (regardless of $\theta$) to get
\begin{equation*}
\int_{\mathbb{R}^N}\big((\hat{\theta}(y))_j-\theta_j\big)\frac{\partial\log f(y;\theta)}{\partial\theta_i}f(y;\theta)dy
=\frac{\partial \theta_j}{\partial \theta_i}
=\left\{
\begin{array}{ll}
1&\mbox{if } i=j\\
0&\mbox{if } i\neq j
\end{array}
\right.
.
\end{equation*}
Thus, letting $\nabla_\theta\log f(y;\theta)$ denote the column vector whose $i$th entry is $\frac{\partial\log f(y;\theta)}{\partial\theta_i}$,
we have
\begin{equation*}
I
=\int_{\mathbb{R}^N}\big(\hat{\theta}(y)-\theta\big)\big(\nabla_\theta\log f(y;\theta)\big)^\mathrm{T}f(y;\theta)dy.
\end{equation*}
Equivalently, we have that for all column vectors $a,b\in\mathbb{R}^{2M}$,
\begin{equation*}
a^\mathrm{T}b
=\int_{\mathbb{R}^N}a^\mathrm{T}\big(\hat{\theta}(y)-\theta\big)\big(\nabla_\theta\log f(y;\theta)\big)^\mathrm{T}b~f(y;\theta)dy.
\end{equation*}
Next, we apply the Cauchy-Schwarz inequality in $f$-weighted $L^2$ space to get
\begin{align*}
\big(a^\mathrm{T}b\big)^2
&=\bigg(\int_{\mathbb{R}^N}a^\mathrm{T}\big(\hat{\theta}(y)-\theta\big)\big(\nabla_\theta\log f(y;\theta)\big)^\mathrm{T}b~f(y;\theta)dy\bigg)^2\\
&\leq\bigg(\int_{\mathbb{R}^N}a^\mathrm{T}\big(\hat{\theta}(y)-\theta\big)\big(\hat{\theta}(y)-\theta\big)^\mathrm{T}a~f(y;\theta)dy\bigg)\bigg(\int_{\mathbb{R}^N}b^\mathrm{T}\big(\nabla_\theta\log f(y;\theta)\big)\big(\nabla_\theta\log f(y;\theta)\big)^\mathrm{T}b~f(y;\theta)dy\bigg)\\
&=\big(a^\mathrm{T}C(\hat{\theta})a\big)\big(b^\mathrm{T}J(\theta)b\big),
\end{align*}
where the last step follows from pulling vectors out of integrals.
At this point, we take $b:=[\widetilde{J}(\theta)^{-1}\tilde{a}; 0]$, where $\tilde{a}$ is the first $2M-1$ entries of $a$.
Then
\begin{equation}
\label{eq.CRLB to divide}
\big(\tilde{a}^\mathrm{T}\widetilde{J}(\theta)^{-1}\tilde{a}\big)^2
=\big(a^\mathrm{T}b\big)^2
\leq\big(a^\mathrm{T}C(\hat{\theta})a\big)\big(b^\mathrm{T}J(\theta)b\big)
=\big(a^\mathrm{T}C(\hat{\theta})a\big)\big(\tilde{a}^\mathrm{T}\widetilde{J}(\theta)^{-1}\tilde{a}\big).
\end{equation}
At this point, we note that since the last (complex) entry of $\theta\in\Omega$ is necessarily positive, then as a $2M$-dimensional real vector, the last entry is necessarily zero, and furthermore every unbiased estimator $\hat{\theta}$ in $\Omega$ will also vanish in the last entry.
It follows that the last row and column of $C(\hat{\theta})$ are both zero.
Furthermore, since $\widetilde{J}(\theta)^{-1}$ is positive by Lemma~\ref{lemma.J is invertible}, division in \eqref{eq.CRLB to divide} gives
\begin{equation*}
\big(\tilde{a}^\mathrm{T}\widetilde{J}(\theta)^{-1}\tilde{a}\big)
\leq\big(a^\mathrm{T}C(\hat{\theta})a\big)
=\big(\tilde{a}^\mathrm{T}\widetilde{C}(\hat{\theta})\tilde{a}\big),
\end{equation*}
from which the result follows.
\qquad
\end{proof}

\section*{Appendix}

Here, we verify that we can differentiate under the integral sign in the proof of Theorem~\ref{thm.complex random stability}.

\begin{lemma}
\label{lemma.differentiation under the integral sign}
Consider the probability density function defined by
\begin{equation*}
f(y;\theta)
=\frac{1}{(2\pi \sigma^2)^{N/2}}e^{-\|y-\mathcal{A}(\theta)\|^2/2\sigma^2}
\qquad
\forall y\in\mathbb{R}^N.
\end{equation*}
Then for every function $g\colon\mathbb{R}^N\rightarrow\mathbb{R}$ with finite second moment
\begin{equation*}
\int_{\mathbb{R}^N}g(y)^2f(y;\theta)dy
<\infty
\qquad
\forall \theta\in\Omega,
\end{equation*}
we can differentiate under the integral sign:
\begin{equation*}
\frac{\partial}{\partial\theta_i}\int_{\mathbb{R}^N}g(y)f(y;\theta)dy
=\int_{\mathbb{R}^N}g(y)\frac{\partial}{\partial\theta_i}f(y;\theta)dy.
\end{equation*}
\end{lemma}

\begin{proof}
First, we adapt the proof of Lemma~5.14 in~\cite{LehmannC:98} to show that it suffices to find a function $b(y;\theta)$ with finite second moment such that, for some $\varepsilon>0$,
\begin{equation}
\label{eq.b requirement}
\bigg|\frac{f(y;\theta+z\delta_i)-f(y;\theta)}{zf(y;\theta)}\bigg|\leq b(y;\theta)
\qquad
\forall y\in\mathbb{R}^{N}, \theta\in\Omega, |z|<\varepsilon, z\neq0
\end{equation}
where $\delta_i$ denotes the $i$th identity basis element in $\mathbb{R}^{2M}$.
Indeed, by applying the Cauchy-Schwarz inequality over $f$-weighted $L^2$ space, we have
\begin{equation*}
\int_{\mathbb{R}^N}|g(y)|b(y;\theta)f(y;\theta)dy
\leq\bigg(\int_{\mathbb{R}^N}g(y)^2f(y;\theta)dy
\bigg)^{1/2}\bigg(\int_{\mathbb{R}^N}b(y;\theta)^2f(y;\theta)dy
\bigg)^{1/2}
<\infty
\end{equation*}
and so the dominated convergence theorem gives
\begin{align*}
\int_{\mathbb{R}^N}g(y)\frac{\partial}{\partial\theta_i}f(y;\theta)dy
&=\int_{\mathbb{R}^N}\lim_{z\rightarrow0}\bigg(g(y)\frac{f(y;\theta+z\delta_i)-f(y;\theta)}{zf(y;\theta)}\bigg)f(y;\theta)dy\\
&=\lim_{z\rightarrow0}\int_{\mathbb{R}^N}\bigg(g(y)\frac{f(y;\theta+z\delta_i)-f(y;\theta)}{zf(y;\theta)}\bigg)f(y;\theta)dy\\
&=\lim_{z\rightarrow0}\frac{1}{z}\bigg(\int_{\mathbb{R}^N}g(y)f(y;\theta+z\delta_i)dy-\int_{\mathbb{R}^N}g(y)f(y;\theta)dy\bigg)\\
&=\frac{\partial}{\partial\theta_i}\int_{\mathbb{R}^N}g(y)f(y;\theta)dy.
\end{align*}
In pursuit of such a function $b(y;\theta)$, we first use the triangle and Cauchy-Schwarz inequalities to get
\begin{align}
\bigg|\frac{f(y;\theta+z\delta_i)-f(y;\theta)}{zf(y;\theta)}\bigg|
\nonumber
&=\frac{1}{|z|}\Big|e^{-\frac{1}{2\sigma^2}\big(\|y-\mathcal{A}(\theta+z\delta_i)\|^2-\|y-\mathcal{A}(\theta)\|^2\big)}-1\Big|\\
\nonumber
&=\frac{1}{|z|}\Big|e^{-\frac{1}{2\sigma^2}\big(\|\mathcal{A}(\theta+z\delta_i)\|^2-\|\mathcal{A}(\theta)\|^2-2\langle y,\mathcal{A}(\theta+z\delta_i)-\mathcal{A}(\theta)\rangle\big)}-e^{\frac{1}{\sigma^2}\langle y,\mathcal{A}(\theta+z\delta_i)-\mathcal{A}(\theta)\rangle}
+e^{\frac{1}{\sigma^2}\langle y,\mathcal{A}(\theta+z\delta_i)-\mathcal{A}(\theta)\rangle}-1\Big|\\
\nonumber
&\leq\frac{1}{|z|}\bigg(e^{\frac{1}{\sigma^2}\langle y,\mathcal{A}(\theta+z\delta_i)-\mathcal{A}(\theta)\rangle}\Big|e^{-\frac{1}{2\sigma^2}\big(\|\mathcal{A}(\theta+z\delta_i)\|^2-\|\mathcal{A}(\theta)\|^2\big)}-1\Big|+\Big|e^{\frac{1}{\sigma^2}\langle y,\mathcal{A}(\theta+z\delta_i)-\mathcal{A}(\theta)\rangle}-1\Big|\bigg)\\
\label{eq.diff int to bound}
&\leq\frac{1}{|z|}\bigg(e^{\frac{1}{\sigma^2}\|y\|\|\mathcal{A}(\theta+z\delta_i)-\mathcal{A}(\theta)\|}\Big|e^{-\frac{1}{2\sigma^2}\big(\|\mathcal{A}(\theta+z\delta_i)\|^2-\|\mathcal{A}(\theta)\|^2\big)}-1\Big|+\Big|e^{\frac{1}{\sigma^2}\|y\|\|\mathcal{A}(\theta+z\delta_i)-\mathcal{A}(\theta)\|}-1\Big|\bigg),
\end{align}
Denote $c(z;\theta):=\frac{1}{\sigma^2}\|\mathcal{A}(\theta+z\delta_i)-\mathcal{A}(\theta)\|$.
Since $(e^{st}-1)/t\leq se^{st}$ whenever $s,t\geq0$, we then have
\begin{equation*}
\frac{|e^{c(z;\theta)\|y\|}-1|}{|z|}
=\frac{c(z;\theta)}{|z|}\cdot\frac{e^{c(z;\theta)\|y\|}-1}{c(z;\theta)}
\leq\frac{c(z;\theta)}{|z|}\|y\|e^{c(z;\theta)\|y\|}.
\end{equation*}
Also by l'Hospital's rule, there exist continuous functions $C_1$ and $C_2$ on the real line such that
\begin{equation*}
C_1(z;\theta)
=\frac{c(z;\theta)}{|z|},
\qquad
C_2(z;\theta)
=\frac{1}{|z|}\Big|e^{-\frac{1}{2\sigma^2}\big(\|\mathcal{A}(\theta+z\delta_i)\|^2-\|\mathcal{A}(\theta)\|^2\big)}-1\Big|,
\qquad
\forall z\neq0.
\end{equation*}
Thus, continuing \eqref{eq.diff int to bound} gives
\begin{equation*}
\bigg|\frac{f(y;\theta+z\delta_i)-f(y;\theta)}{zf(y;\theta)}\bigg|
\leq\Big(C_1(z;\theta)\|y\|+C_2(z;\theta)\Big)e^{c(z;\theta)\|y\|}.
\end{equation*}
%
%
%
Now for a fixed $\varepsilon$, take $C_j(\theta):=\sup_{|z|<\varepsilon}C_j(z;\theta)$ and $c(\theta):=\sup_{|z|<\varepsilon}c(z;\theta)$, and define 
\begin{equation*}
b(y;\theta):=\Big(C_1(\theta)\|y\|+C_2(\theta)\Big)e^{c(\theta)\|y\|}.
\end{equation*}
Since $C_j(\theta)$ and $c(\theta)$ are suprema of continuous functions over a bounded set, these are necessarily finite for all $\theta\in\Omega$.
As such, our choice for $b$ satisfies \eqref{eq.b requirement}.
It remains to verify that $b$ has a finite second moment.
To this end, let $B(R(\theta))$ denote the ball of radius $R(\theta)$ centered at the origin (we will specify $R(\theta)$ later).
Then
\begin{align}
\nonumber
\int_{\mathbb{R}^N}b(y;\theta)^2f(y;\theta)dy
&=\int_{B(R(\theta))}b(y;\theta)^2f(y;\theta)dy+\int_{\mathbb{R}^N\setminus B(R(\theta))}b(y;\theta)^2f(y;\theta)dy\\
\nonumber
&\leq \Big(C_1(\theta)R(\theta)+C_2(\theta)\Big)^2e^{2c(\theta)R(\theta)}\\
\label{eq.bound second moment of b}
&\qquad+\frac{1}{(2\pi \sigma^2)^{N/2}}\int_{\mathbb{R}^N\setminus B(R(\theta))}\Big(C_1(\theta)\|y\|+C_2(\theta)\Big)^2e^{2c(\theta)\|y\|-\frac{1}{2\sigma^2}\|y-\mathcal{A}(\theta)\|^2}dy.
\end{align}
From here, we note that whenever $\|y\|\geq 2\|\mathcal{A}(\theta)\|+8\sigma^2c(\theta)$, we have
\begin{align*}
\|y-\mathcal{A}(\theta)\|^2
&\geq\|y\|^2-2\|y\|\|\mathcal{A}(\theta)\|+\|\mathcal{A}(\theta)\|^2\\
&\geq\Big(2\|\mathcal{A}(\theta)\|+8\sigma^2c(\theta)\Big)\|y\|-2\|y\|\|\mathcal{A}(\theta)\|+\|\mathcal{A}(\theta)\|^2\\
&\geq8\sigma^2c(\theta)\|y\|.
\end{align*}
Rearranging then gives $2c(\theta)\|y\|\leq\frac{1}{4\sigma^2}\|y-\mathcal{A}(\theta)\|^2$.
Also let $h(\theta)$ denote the larger root of the polynomial
\begin{equation*}
p(x;\theta):=2C_1(\theta)^2\Big(x^2-2\|\mathcal{A}(\theta)\|x+\|\mathcal{A}(\theta)\|^2\Big)-\Big(C_1(\theta)x+C_2(\theta)\Big)^2,
\end{equation*}
and take $h(\theta):=0$ when the roots of $p(x;\theta)$ are not real.
(Here, we are assuming that $C_1>0$, but the proof that \eqref{eq.bound second moment of b} is finite when $C_1=0$ quickly follows from the $C_1>0$ case.)
Then $(C_1(\theta)\|y\|+C_2(\theta))^2\leq 2C_1(\theta)^2\|y-\mathcal{A}(\theta)\|^2$ whenever $\|y\|\geq h(\theta)$, since by the Cauchy-Schwarz inequality,
\begin{equation*}
2C_1(\theta)^2\|y-\mathcal{A}(\theta)\|^2-\Big(C_1(\theta)\|y\|+C_2(\theta)\Big)^2
\geq p(\|y\|;\theta)
\geq 0,
\end{equation*}
where the last step follows from the fact that $p(x;\theta)$ is concave up.
Now we continue by taking $R(\theta):=\max\{2\|\mathcal{A}(\theta)\|+8\sigma^2c(\theta),h(\theta)\}$:
\begin{align*}
&\int_{\mathbb{R}^N\setminus B(R(\theta))}\Big(C_1(\theta)\|y\|+C_2(\theta)\Big)^2e^{2c(\theta)\|y\|-\frac{1}{2\sigma^2}\|y-\mathcal{A}(\theta)\|^2}dy\\
&\qquad\leq \int_{\mathbb{R}^N\setminus B(R(\theta))}2C_1(\theta)^2\|y-\mathcal{A}(\theta)\|^2e^{-\frac{1}{4\sigma^2}\|y-\mathcal{A}(\theta)\|^2}dy\\
&\qquad\leq \big(2\pi(\sqrt{2}\sigma)^2\big)^{N/2}\cdot2C_1(\theta)^2\int_{\mathbb{R}^N}\|x\|^2\frac{1}{(2\pi(\sqrt{2}\sigma)^2)^{N/2}}e^{-\|x\|^2/2(\sqrt{2}\sigma)^2}dx,
\end{align*}
where the last step comes from integrating over all of $\mathbb{R}^N$ and changing variables $y-\mathcal{A}(\theta)\mapsto x$.
This last integral calculates the expected squared length of a vector in $\mathbb{R}^N$ with independent $\mathcal{N}(0,2\sigma^2)$ entries, which is $2N\sigma^2$.
Thus, substituting into \eqref{eq.bound second moment of b} gives that $b$ has a finite second moment.
\qquad
\end{proof}

\section*{Acknowledgments}
The authors thank Irene Waldspurger and Profs.\ Bernhard G.\ Bodmann, Matthew Fickus, Thomas Strohmer and Yang Wang for insightful discussions, and the Erwin Schr\"{o}dinger International Institute for Mathematical Physics for hosting a workshop on phase retrieval that helped solidify some of the ideas in this paper.
A.\ S.\ Bandeira was supported by NSF DMS-0914892, and J.\ Cahill was supported by NSF 1008183, NSF ATD 1042701, and AFOSR DGE51: FA9550-11-1-0245.
The views expressed in this article are those of the authors and do not reflect the official policy or position of the United States Air Force, Department of Defense, or the U.S.~Government.

\end{document}